\documentclass[11pt]{article}

\usepackage{mathtools, tikz}
\usepackage{csquotes}
\usepackage{xpatch}

\usepackage{amsthm, amsmath, amssymb, amsfonts}
\usepackage{url}
\usepackage[inline,shortlabels]{enumitem}
\usepackage[T1]{fontenc} % for \{ in \texttt
\usepackage[czech,english]{babel} % This will make Stovicek look correct
\usepackage[margin = 1in]{geometry}

\usepackage[
    sorting=nyt,
    maxbibnames=99,
    backref,
    backrefstyle=none % do not compress the list
]{biblatex}

\usepackage[
    colorlinks,
    linkcolor={red!80!black},
    citecolor={blue!80!black},
    urlcolor={green!80!black}
]{hyperref}
\usepackage{cleveref}

\newcommand{\footremember}[2]{%
    \footnote{#2}
    \newcounter{#1}
    \setcounter{#1}{\value{footnote}}%
}
\newcommand{\footrecall}[1]{%
    \footnotemark[\value{#1}]%
}

\setenumerate{label=(\arabic*)}

\newcommand{\Mc}{{\normalfont\rmfamily M}} % M-convexity

\counterwithin{equation}{section}

\crefname{subsection}{subsection}{subsections}
\Crefname{conjecture}{Conjecture}{Conjectures}

\DeclareFieldFormat{url}{\href{#1}{Link}}
\DeclareFieldFormat{doi}{\href{https://dx.doi.org/#1}{\textsc{doi}}}
\DeclareFieldFormat{eprint:arxiv}{\href{https://arxiv.org/abs/#1}{arXiv:\allowbreak{#1}}}
\DeclareFieldFormat{eprint}{\href{https://arxiv.org/abs/#1}{arXiv:\allowbreak{#1}}}
\AtEveryBibitem{\clearfield{issn}\clearfield{isbn}\clearlist{language}}

\renewbibmacro*{doi+eprint+url}{%
    \printfield{doi}%
    \newunit\newblock%
    \iftoggle{bbx:eprint}{%
        \usebibmacro{eprint}%
    }{}%
    \newunit\newblock%
    \iffieldundef{doi}{%
        \usebibmacro{url+urldate}%
    }{}%
}

\NewBibliographyString{toappear}
\DefineBibliographyStrings{english}{%
  toappear = {to appear},
}

\renewbibmacro*{in:}{%
  \iffieldundef{pubstate}
    {}
    {\printfield{pubstate}%
     \setunit{\addspace}%
     \clearfield{pubstate}}%
  \bibstring{in}%
  \printunit{\intitlepunct}}

\DefineBibliographyStrings{english}{
    backrefpage={},
    backrefpages={}
}

\DeclareFieldFormat{backrefparens}{\addperiod{\scriptsize{#1}}}
\xpatchbibmacro{pageref}{parens}{backrefparens}{}{}

\newtheorem{theorem}{Theorem}[section]
\newtheorem{proposition}[theorem]{Proposition}
\newtheorem{lemma}[theorem]{Lemma}
\newtheorem{corollary}[theorem]{Corollary}
\newtheorem{conjecture}[theorem]{Conjecture}
\newtheorem{question}[theorem]{Question}

\newtheorem{claim}[theorem]{Claim}

\theoremstyle{definition}

\theoremstyle{remark}

\crefname{claim}{claim}{claims}

\makeatletter
\newenvironment{claimproof}[1][Proof]{\par
    \pushQED{\qed}%
    
    \normalfont \topsep6\p@\@plus6\p@\relax
    \trivlist
    \item[\hskip\labelsep
    \textit{#1}\@addpunct{.}~]\ignorespaces
}{%
    \popQED\endtrivlist\@endpefalse
}
\makeatother

\newcommand{\RR}{\mathbb{R}}
\newcommand{\NN}{\mathbb{N}}

\newcommand{\ba}{\mathbf{a}}
\newcommand{\bb}{\mathbf{b}}
\newcommand{\bx}{\mathbf{x}}
\newcommand{\by}{\mathbf{y}}

\newcommand{\be}{\mathbf{e}}

\DeclareMathOperator{\Hom}{Hom}
\newcommand{\bhom}{\operatorname{hom}_\mathrm{b}}
\newcommand{\bHom}{\operatorname{Hom}_\mathrm{b}}
\DeclareMathOperator{\supp}{supp}

\newcommand{\Hess}{\mathop{}\!\mathcal{H}} % Hessian
\newcommand{\one}{\mathbf{1}} % True-False indicator

\DeclarePairedDelimiter{\abs}{\lvert}{\rvert}

\newcommand{\defeq}{\coloneqq}

\DeclareRobustCommand{\stirling}{\genfrac\{\}{0pt}{}}

\newcommand{\blank}{{-}}
\newcommand{\HH}{\mathcal{H}}
\newcommand{\K}{\mathcal{K}}

\makeatletter
\newlength{\negph@wd}
\DeclareRobustCommand{\negphantom}[1]{%
  \ifmmode
    \mathpalette\negph@math{#1}%
  \else
    \negph@do{#1}%
  \fi
}
\newcommand{\negph@math}[2]{\negph@do{$\m@th#1#2$}}
\newcommand{\negph@do}[1]{%
  \settowidth{\negph@wd}{#1}%
  \hspace*{-\negph@wd}%
}
\makeatother

\newcommand{\swapped}[2]{{#1}^{\bowtie}_{#2}}

\tikzset{every loop/.style={}}
\tikzset{inline vertex/.style={draw, circle, fill=black, minimum size=2pt, inner sep=0pt}}

\newcommand{\indepg}{\!%
    \begin{tikzpicture}[nodes=inline vertex]
        \node (A) at (0,0) {};
        \node (B) at (0.42,0) {};
        \draw[looseness=15] (A) to[out=125, in=55] (A);
        \draw (A) -- (B);
    \end{tikzpicture}%
    \:\!%
}

\newcommand{\pathtwo}{%
    \begin{tikzpicture}[nodes=inline vertex]
        \node (A) at (0, 2.5pt) {};
        \node (B) at (10pt, 0) {};
        \node (C) at (10pt, 5pt) {};
        \draw (A) -- (B);
        \draw (A) -- (C);
    \end{tikzpicture}%
    \:\!%
}

\newcommand{\paththree}{%
    \begin{tikzpicture}[nodes=inline vertex]
        \node (A) at (0,0) {};
        \node (B) at (0.42,0) {};
        \node (C) at (0.42,0.2) {};
        \node (D) at (0,0.2) {};
        \draw (A) -- (B);
        \draw (A) -- (C);
        \draw (C) -- (D);
    \end{tikzpicture}%
    \:\!%
}

\newcommand{\twoedges}{%
    \begin{tikzpicture}[nodes=inline vertex]
        \node (A) at (0,0) {};
        \node (B) at (0.42,0) {};
        \node (C) at (0.42,0.2) {};
        \node (D) at (0,0.2) {};
        \draw (A) -- (B);
        \draw (C) -- (D);
    \end{tikzpicture}%
    \:\!%
}

\newcommand{\triangleaddedge}{%
    \begin{tikzpicture}[nodes=inline vertex]
        \node (A) at (0,0) {};
        \node (B) at (0.42,0) {};
        \node (C) at (0.42,0.2) {};
        \node (D) at (0,0.2) {};
        \draw (A) -- (B);
        \draw (A) -- (C);
        \draw (B) -- (C);
        \draw (C) -- (D);
    \end{tikzpicture}%
    \:\!%
}

\newcommand{\twoloops}{\!%
    \begin{tikzpicture}[nodes=inline vertex]
        \node (A) at (0,0) {};
        \node (B) at (0.42,0) {};
        \draw[looseness=15] (A) to[out=125, in=55] (A);
        \draw[looseness=15] (B) to[out=125, in=55] (B);
    \end{tikzpicture}%
    \!%
}

\newcommand{\loopplusvtx}{\!%
    \begin{tikzpicture}[nodes=inline vertex]
        \node (A) at (0,0) {};
        \node (B) at (0.42,0) {};
        \draw[looseness=15] (A) to[out=125, in=55] (A);
    \end{tikzpicture}%
    \:\!%
}

\newcommand{\loopplusedge}{\!%
    \begin{tikzpicture}[nodes=inline vertex]
        \node (A) at (0,0) {};
        \node (B) at (0.42,0) {};
        \node (C) at (0.42,0.2) {};
        \draw (B) -- (C);
        \draw[looseness=15] (A) to[out=125, in=55] (A);
    \end{tikzpicture}%
    \:\!%
}

\newcommand{\loopplusloopededge}{%
    \begin{tikzpicture}[nodes=inline vertex]
        \node (A) at (0,0) {};
        \node (B) at (0.42,0) {};
        \node (C) at (0.42,0.2) {};
        \draw (B) -- (C);
        \draw[looseness=15] (A) to[out=125, in=55] (A);
        \draw[looseness=15] (C) to[out=325, in=35] (C);
    \end{tikzpicture}%
}

\newcommand{\loopedpathtwo}{\!%
    \begin{tikzpicture}[nodes=inline vertex]
        \node (A) at (0,0) {};
        \node (B) at (0.42,0) {};
        \node (C) at (0.42,0.2) {};
        \draw (A) -- (C);
        \draw (B) -- (C);
        \draw[looseness=15] (A) to[out=125, in=55] (A);
    \end{tikzpicture}%
    \:\!%
}

\newcommand{\twoloopspathtwo}{\!%
    \begin{tikzpicture}[nodes=inline vertex]
        \node (A) at (0,0) {};
        \node (B) at (0.42,0) {};
        \node (C) at (0.42,0.2) {};
        \draw (A) -- (C);
        \draw (B) -- (C);
        \draw[looseness=15] (C) to[out=325, in=35] (C);
        \draw[looseness=15] (A) to[out=125, in=55] (A);
    \end{tikzpicture}%
}

\title{Counting homomorphisms in antiferromagnetic graphs via Lorentzian polynomials}
    \author{
    Joonkyung Lee%
    \footremember{Yonsei}{
        Department of Mathematics, Yonsei University, Seoul, South Korea. Research supported by Samsung STF Grant SSTF-BA2201-02, the Yonsei University Research Fund 2023-22-0125 and the National Research Foundation of Korea (NRF) grant MSIT NRF-2022R1C1C1010300. Email: \texttt{\{joonkyunglee, jaehyeonseo\}@yonsei.ac.kr}.
    }
    \and
    Jaeseong Oh%
    \thanks{
        June E Huh Center for Mathematical Challenges, Korea Institute for Advanced Study, South Korea.
        Research supported by a KIAS Individual Grant (HP083401) and the National Research Foundation of Korea (NRF) grant MSIT NRF-2022R1C1C1010300. Email: \texttt{jsoh@kias.re.kr}.
    }
    \and
    Jaehyeon Seo%
    \footrecall{Yonsei}
}
\date{}

\begin{document}
\maketitle

\begin{abstract}
    An edge-weighted graph $G$, possibly with loops, is said to be \emph{antiferromagnetic} if it has nonnegative weights and at most one positive eigenvalue, counting multiplicities.
    The number of graph homomorphisms from a graph $H$ to an antiferromagnetic graph $G$ generalises various important parameters in graph theory, including the number of independent sets and proper vertex-colourings, as well as their relaxations in statistical physics.

    We obtain homomorphism inequalities for various graphs $H$ and antiferromagnetic graphs~$G$ of the form
    \[
        \lvert\operatorname{Hom}(H,G)\rvert^2 \leq \lvert\operatorname{Hom}(H\times K_2,G)\rvert,
    \]
    where $H\times K_2$ denotes the tensor product of $H$ and $K_2$.
    Firstly, we show that the inequality holds for any $H$ obtained by blowing up vertices of a bipartite graph into complete graphs and any antiferromagnetic $G$. In particular, one can take $H=K_{d+1}$, which already implies a new result for the Sah--Sawhney--Stoner--Zhao conjecture on the maximum number of $d$-regular graphs in antiferromagnetic graphs.
    Secondly, the inequality also holds for $G=K_q$ and those $H$ obtained by blowing up vertices of a bipartite graph into complete multipartite graphs, paths or even cycles.
    
    Both results can be seen as the first progress towards Zhao's conjecture on $q$-colourings, which states that the inequality holds for any $H$ and $G=K_q$, after his own work. Our method leverages on the emerging theory of Lorentzian polynomials due to Br\"and\'en and Huh and log-concavity of the list colourings of bipartite graphs, which may be of independent interest.
\end{abstract}

\section{Introduction}
For graphs $G$ and $H$, a \emph{homomorphism} from $H$ to $G$ is a vertex map that preserves adjacency. There are numerous concepts in graph theory that rephrase in terms of homomorphisms, which include two fundamental examples: independent sets and proper vertex-colourings with $q$ colours, or simply \emph{$q$-colourings.}
Indeed, there is a natural bijection from independent sets of a graph $H$ and homomorphisms from $H$ to $G=\indepg$ and each $q$-colouring of $H$ corresponds to a homomorphism from $H$ to $G=K_q$, the complete graph on $q$ vertices.

This correspondence translates extremal problems on the number of independent sets or $q$-colourings to homomorphism inequalities.
For instance, the Kahn--Zhao theorem \cite{kahn2001entropy,zhao2010number} states that the complete bipartite graph $K_{d,d}$ has the maximum number of independent sets amongst $d$-regular graphs. More precisely, if $H$ is $d$-regular, then
\begin{equation}\label{eq:indep_count_in_d-reg}
    \hom(H,\indepg)^{1/v(H)} \leq \hom(H\times K_2,\indepg)^{1/(2v(H))}\leq \hom(K_{d,d},\indepg)^{1/(2d)},
\end{equation}
where $H\times K_2$ denotes the tensor product of $H$ and $K_2$ and $\hom(H,G)=\abs{\Hom(H,G)}$ is the number of homomorphisms from $H$ to $G$ and $v(H)$ is the number of vertices in $H$.
We note that the first inequality is due to Zhao's `bipartite swapping trick' in~\cite{zhao2010number} and the second is due to Kahn~\cite{kahn2001entropy}.

There have been a lot of exciting developments along these lines of research for the last decade or two, which touch upon information theory~\cite{friedgut2004hypergraphs} and statistical physics~\cite{davies2017independent}. For a survey on the topic, see \cite{zhao2017extremal}.
Especially, the second inequality in~\eqref{eq:indep_count_in_d-reg} has been generalised multiple times~\cite{galvin2004weighted,sah2019number,sah2020reverse}, whose strongest form is the `reverse Sidorenko' inequality of Sah, Sawhney, Stoner, and Zhao~\cite{sah2020reverse}.
\begin{theorem}[Theorem~1.9~in~\cite{sah2020reverse}]\label{thm:SSSZ}
    Let $H$ be a $d$-regular triangle-free graph and let $G$ be a graph possibly with loops. Then
    \begin{equation}\label{eq:reverseSidorenko}
        \hom(H,G)^{1/v(H)} \leq \hom(K_{d,d},G)^{1/(2d)}.
    \end{equation}
\end{theorem}
We note that the original statement of~\cite[Theorem~1.9]{sah2020reverse} allows distinct degrees in $H$, which generalises~\Cref{thm:SSSZ} to arbitrary triangle-free graphs. On the other hand, the triangle-freeness condition is essential in the sense that, for every $H$ that contains a triangle, there exists $G$ that breaks the inequality~\eqref{eq:reverseSidorenko}. 
In contrast, the Kahn--Zhao theorem requires no condition on $H$ while an extremely specific target graph $G=\indepg$ is chosen.
%This is in part due to the fact that there is no inequality that replaces the target graph $\indepg$ by an arbitrary $G$.

 In~\cite{sah2020reverse}, Sah, Sawhney, Stoner and Zhao also showed that letting $G=K_q$ in~\eqref{eq:reverseSidorenko} gives another example where triangle-freeness of $H$ is not necessary.
% letting $G=K_q$ in~\eqref{eq:reverseSidorenko} gives
\begin{theorem}[Theorem~1.7~in~\cite{sah2020reverse}]\label{thm:SSSZ_q}
    For any $d$-regular graph $H$ and $q\geq 2$,
    \[
        \hom(H,K_q)^{1/v(H)}\leq \hom(K_{d,d},K_q)^{1/(2d)}.
    \] 
\end{theorem}
\noindent That is, complete bipartite graphs have maximum number of $q$-colourings amongst $d$-regular graphs. In fact, \cite[Theorem~1.7]{sah2020reverse} proves more general statement that allows some graphs other than $K_q$, which will be discussed later with~\Cref{conj:SSSZ}, and distinct degrees in~$H$.
It is worth mentioning that, before~\cite{sah2020reverse}, the case $d=3$ proven in~\cite{davies2018potts} was the only known case, which highlights the significance of the result and the difficulty of proving such inequalities.

One might wonder whether \Cref{thm:SSSZ_q} follows from an analogous strategy to the Kahn--Zhao theorem. More precisely, can we simply replace {\indepg} by $K_q$ in the two inequalities~\eqref{eq:indep_count_in_d-reg}, given that the second already holds by~\Cref{thm:SSSZ}?
In fact, it remains unknown whether the inequality obtained by replacing {\indepg} by $K_q$ in the first inequality in~\eqref{eq:indep_count_in_d-reg} is true, although Zhao~\cite{zhao2011bipartite} conjectured so more than a decade ago.
\begin{conjecture}[Conjecture~5.1 in~\cite{zhao2011bipartite}]\label{conj:zhao}
    For any graph $H$ and $q\geq 2$,
    \begin{equation}\label{eq:bipartite_swap_q}
        \hom(H,K_q)^2 \leq \hom(H\times K_2,K_q).
    \end{equation}
\end{conjecture}
As $H\times K_2$ is always bipartite, \Cref{conj:zhao} combined with~\Cref{thm:SSSZ} implies~\Cref{thm:SSSZ_q}.\footnote{It also implies the generalisation of~\Cref{thm:SSSZ_q} for irregular graphs, which we do not state here for simplicity.}
If $H$ is bipartite, then $H\times K_2$ is isomorphic to two vertex-disjoint copies of $H$, so the equality always holds in~\eqref{eq:bipartite_swap_q}.
Zhao~\cite{zhao2011bipartite} settled the conjecture for large enough $q=q(H)$ for every fixed $H$, but there has been no progress on the conjecture since then, to the best of our knowledge.

Our first main result is to find new graphs $H$ that satisfy~\eqref{eq:bipartite_swap_q} for arbitrary $q\geq 2$.
Let $\HH$ be a class of graphs.
%that contains the single vertex graph $K_1$. 
Then a graph $H$ is an \emph{$\HH$-blow-up} of a $k$-vertex graph $F$ if it is obtained by replacing all vertices $v_1,v_2,\dots,v_k$ of $F$ by $H_1,H_2,\dots, H_k$ in $\HH$, respectively, and the edges of the form $v_iv_j$ by a complete bipartite graph $K_{s,t}$, where $s=v(H_i)$ and $t=v(H_j)$.
\begin{theorem}\label{thm:main2}
    Let $\HH$ be the class of graphs that consists of all complete multipartite graphs, even cycles, and paths. If $H$ is an $\HH$-blow-up of a bipartite graph, then for any $q\geq 2$,
    \[
        \hom(H,K_q)^2 \leq \hom(H\times K_2,K_q).
    \]
\end{theorem}
Complete $1$-partite graphs are empty graphs on isolated vertices, so an $\HH$-blow-up above also includes the blow-ups in the usual sense. In particular, blowing up a vertex by $K_1$ keeps it unchanged.
Our proof uses `log-submodularity' inequalities for certain list colourings of complete multipartite graphs, even cycles, and paths. In fact, complete graphs satisfy much stronger inequalities than these, which we will revisit with the second main result.

\medskip

Having seen the similarities between the Kahn--Zhao theorem and~\Cref{thm:SSSZ_q}, one may wonder whether there is a common generalisation of \(G=\indepg\) and \(G=K_q\) that allows us to obtain an inequality of the form~\eqref{eq:reverseSidorenko} for arbitrary $d$-regular graphs $H$, not necessarily triangle-free.
Sah, Sawhney, Stoner, and Zhao conjectured that $G$, as a symmetric matrix, having at most one positive eigenvalue may be the correct condition to add on.% \comment{Is it worth emphasizing that the triangle-freeness is dropped?}
\begin{conjecture}[Conjecture~1.16 in~\cite{sah2020reverse}]\label{conj:SSSZ}
    Let $H$ be a $d$-regular graph and let $G$ be a graph possibly with loops that has at most one positive eigenvalue. Then
    \[
        \hom(H,G)^{1/v(H)} \leq \hom(K_{d,d},G)^{1/(2d)}.
    \]
\end{conjecture}
In~\cite{sah2020reverse}, the conjecture is confirmed for the particular case of \emph{semiproper colourings}, i.e., complete graphs $G$ possibly with loops. This can be seen as a common generalisation of {\indepg} and $K_q$, albeit weaker than the conjecture itself.

Identifying a graph $G$ by its adjacency matrix naturally generalises to an arbitrary symmetric matrix with nonnegative entries, or equivalently a \emph{weighted graph}, whose edges are weighted by the corresponding entry in the symmetric matrix. Then the homomorphism count $\hom(H,G)$ is also weighted in the sense that
\[
    \hom(H,G) = \sum_{\phi\in\Hom(H,G)}\prod_{uv\in E(H)}G(\phi(u),\phi(v)),
\]
where $G(x,y)$ denotes the corresponding entry of the edge $xy\in E(G)$.
In fact, \Cref{conj:SSSZ} was already stated in terms of weighted graphs $G$  in~\cite{sah2020reverse}.

A good motivation to consider weighted graphs with at most one positive eigenvalue as a common generalisation of {\indepg} and $K_q$ comes from statistical physics, where such a weighted graph is called \emph{antiferromagnetic}, a term that originates from the fundamental Ising and Potts models~\cite{galanis2014inapproximability}.
To elaborate, the partition functions of antiferromagnetic Potts models correspond to $\hom(H,G)$ with an antiferromagnetic matrix $G$ with $e^{-\beta}$ on the diagonal and $1$ elsewhere.
Furthermore, given a suitable adjustment $G$ of the adjacency matrix of $\indepg$\hspace{0.24mm}, $\hom(H,G)$ corresponds to the partition function of hard-core model with rational fugacity $\lambda$, which approximates arbitrary hard-core models.

In what follows, an antiferromagnetic graph always means a weighted one.
Our second main result is a homomorphism inequality that generalises~\eqref{eq:bipartite_swap_q} for arbitrary antiferromagnetic target graphs $G$ instead of $K_q$.
\begin{theorem}\label{thm:main}
    Let $G$ be an antiferromagnetic graph and let $\K$ be the class of all complete graphs. If $H$ is a $\K$-blow-up of a bipartite graph $F$, then
    \begin{equation}\label{eq:main}
     \hom(H,G)^{2}\leq \hom(H\times K_2,G).
    \end{equation}
\end{theorem}
In particular, if $F=K_1$, then \eqref{eq:main} gives $\hom(K_d,G)^2\leq \hom(K_d\times K_2,G)$, which is already a new result that confirms~\Cref{conj:SSSZ} for complete graphs $H$.
For another example, $H=K_{r+s}\setminus K_{s}$ is a $\K$-blow-up of $F=K_{r,1}$, so it also satisfies~\eqref{eq:main}.
Recall that, due to~\Cref{thm:SSSZ}, it is enough to prove~\Cref{conj:SSSZ} for $H$ that contains a triangle. \Cref{thm:main} is the first result that confirms~\Cref{conj:SSSZ} for graphs $H$ that contain triangles and arbitrary antiferromagnetic graphs~$G$. As $\hom(H,G)$ for antiferromagnetic graphs $G$ generalises the partition functions of Potts/Ising or of hard-core model,~\Cref{thm:main} obtains a variety of new inequalities.

Antiferromagnetism also appears in mathematical contexts with different names and their generalisations. For example, antiferromagnetism for an $n$-vertex graph $G$ is equivalent to `hyperbolicity' of its $n\times n$ adjacency matrix $A=A_G$, i.e., $(\bx^T A\bx)(\by^T A\by)\leq (\bx^T A\by)^2$ for all $\bx,\by\in\RR^n$.
As a higher-order generalisation of this, Br\"and\'en and Huh~\cite{branden2020lorentzian} developed the theory of \emph{Lorentzian polynomials}, which applies to settle several old and new conjectures that relate to log-concavity of combinatorial and geometric sequences. In particular, the theory explains why log-concavity appears in the mixed volumes and mixed discriminants in Euclidean spaces and, as a consequence, recovers the classical Alexandrov--Fenchel inequality~\cite[Section~7.3]{schneider2014convex}.

Our main idea in proving~\Cref{thm:main} is to apply this emerging theory of Lorentzian polynomials. First, we find a new class of Lorentzian polynomials given by homomorphic counts for $K_d$. We then obtain a `discrete' analogue of the Alexandrov--Fenchel inequality as a corollary, which is a key tool in proving~\Cref{thm:main}.
To the best of our knowledge, this is the first application of the theory of Lorentzian polynomials in extremal combinatorics.
Although we do not explicitly use Lorentzian polynomials in the proof of~\Cref{thm:main2}, the argument still follows the spirit of the theory in the sense that
we make use of an Alexandrov--Fenchel--type correlation inequality for list colourings of paths, cycles, and complete multipartite graphs.

This paper is organised as follows. In~\Cref{sec:prelim}, some standard results on graph homomorphisms and the theory of Lorentzian polynomials are introduced.
In~\Cref{sec:G_chromatic}, we obtain a new family of Lorentzian polynomials that come from homomorphic $K_t$-counts.
As a consequence,~\Cref{thm:main} follows in the subsequent section.
Throughout~\Cref{sec:BS,sec:SBS-in-Kq}, we prove~\Cref{thm:main2} by using log-submodularity inequalities inspired by the Alexandrov--Fenchel-type techniques used in the previous sections.

\section{Preliminaries}\label{sec:prelim}
In what follows, \(n\), \(q\), and \(t\) are positive integers.
Write \(\NN_0\defeq \{0,1,2,\dots\}\) for the set of nonnegative integers and let \([n]\defeq \{1,2,\dots,n\}\).
For a finite set \(I\), let \(\binom{I}{r}\) be the collection of its subsets of size~$r$.

\paragraph{Graphs.}
A \emph{graph} is assumed to have a non-empty edge set, unless specified otherwise; it has no parallel edges but may have loops, especially if it is identified as a symmetric matrix. We use $v(G)$ and $e(G)$ for the number of vertices and edges of a graph $G$, respectively.
A \emph{weighted graph} \(G\) is a graph with positive edge weights on edges and zero weights on non-edges. 
We write \(G(u,v)\) for the corresponding weight on $uv$ by identifying $G$ as its adjacency matrix with corresponding weights. 
A graph without specified edge weights corresponds to a canonical \(\{0,1\}\)-weight. We sometimes use the term \emph{unweighted graph} to emphasise that we consider an ordinary graph with \(\{0,1\}\)-weights only.
The \emph{support} of a weighted graph $G$, denoted by \(\supp(G)\), is the graph \(G'\) where \(V(G')\defeq V(G)\) and \(uv\in E(G')\) if and only if \(G(u,v)>0\).
Let \(N_G(v)\) be the set of neighbours of the vertex \(v\) in \(G\).
For \(U,W\subseteq V(G)\), write \(G[U]\) for the subgraph of \(G\) induced on \(U\) and let $G[U,W]$ be the \emph{bi-induced} graph on $U\cup W$ with edges of the form $e=uv\in E(G)$ with $u\in U$ and $v\in W$.

To recall, a weighted graph is \emph{antiferromagnetic} if it has at most one positive eigenvalue, counting multiplicities. Recall also that an antiferromagnetic graph always refers to a weighted graph.
We say that a symmetric matrix with nonnegative entries is antiferromagnetic if the corresponding weighted graph is antiferromagnetic.
By the Cauchy interlacing theorem, 
vertex deletions, i.e., taking principal minors, preserve antiferromagnetism.
\begin{proposition}\label{prop:induced}
    An induced subgraph of an antiferromagnetic graph is also antiferromagnetic.
\end{proposition}

\paragraph{Discrete mixed volume.}
If a graph $H$ is $q$-colourable, then the \emph{$q$-chromatic symmetric polynomial}, due to Stanley~\cite{stanley1995symmetric}, is defined by
\[
    X_H(\bx)\defeq \sum_{\phi\in\Hom(H,K_q)}
    \prod_{v\in V(H)} x_{\phi(v)}.
\]
We generalise this definition by replacing $K_q$ by an arbitrary weighted graph. Let \(G\) be an $n$-vertex weighted graph on the vertex set \([n]\) and let \(\bx = (x_1,\dots,x_n)\) be an \(n\)-tuple of variables. A \emph{\(G\)-chromatic function of \(H\)} is an $n$-variable homogeneous polynomial
\[
    h_H(\bx;G) = h_H(x_1,\dots,x_n;G) \defeq \sum_{\phi\colon V(H)\to V(G)} \prod_{uv\in E(H)} G(\phi(u),\phi(v)) \prod_{v\in V(H)} x_{\phi(v)},
\]
which has degree $v(H)$.
Note that the coefficient $\prod_{uv\in E(H)} G(\phi(u),\phi(v))$ is nonzero if and only if $\phi$ is a homomorphism from $H$ to the support of $G$.
In particular, the \(K_q\)-chromatic function of $H$ is the $q$-chromatic symmetric polynomial of $H$. We however remark that \(h_H(\bx;G)\) is \emph{not} symmetric in general.
Indeed, $h_H(x_{\sigma(1)},\dots,x_{\sigma(n)};G) = h_H(x_1,\dots,x_n;G)$ holds for an automorphism $\sigma$ of $\supp(G)$ that preserves edge weights, but this may \emph{not} be true for an arbitrary permutation \(\sigma\) on \([n]\).

The \emph{$G$-volume} of $H$ is a multilinear generalisation of the $G$-chromatic function.
Let \(V(H) = [t]\) and let \(\bx_i=(x_{i,1},\dots,x_{i,n})\), \(i=1,2,\dots,t\) be \(n\)-tuples of variables. The \(G\)-volume of \(H\) is a $t$-variable real function on $(\RR^n)^t$ defined by
\[
    V_H(\bx_1,\bx_{2},\dots,\bx_t;G) \defeq \sum_{\phi\colon V(H)\to V(G)} \prod_{uv\in E(H)} G(\phi(u),\phi(v)) \prod_{u\in V(H)} x_{u,\phi(u)}.
\]
In particular, $V_H(\bx,\bx,\dots,\bx;G)=h_H(\bx;G)$. If each $\bx_{i}$ is the indicator vector for a vertex subset $U_i\subseteq V(G)$, then $V_H(\bx_{1}, \dots,\bx_{t};G)$ counts all the homomorphisms $\phi$ that embed the vertex $i\in V(H)$ into $U_i$.
Note that permuting $\bx_1,\bx_2,\dots,\bx_k$ may change the corresponding $G$-volume of $H$, although the $G$-volume of $H=K_t$ remains unchanged.
%Although the $G$-volume is not symmetric in general, we are mostly interested in properties of the $G$-volume of $K_t$, which is a symmetric function. 

%\comment{elaborate that in what sense, it is symmetric.}

Resembling the definition of mixed volume in Euclidean geometry, 
the identity 
\[
    h_H(\lambda_1\bx_1+\dots+\lambda_r\bx_r; G) = \sum_{(i_1,\dots,i_t)\in [r]^t} V_H(\bx_{i_1},\dots,\bx_{i_t}) \lambda_{i_1}\dotsm \lambda_{i_t}
\]
holds for \(\bx_1,\dots,\bx_r\in \RR^n\) and \(\lambda_1,\dots,\lambda_r \in\RR\).
Considering \(\lambda_1,\dots,\lambda_r\) as variables, we see
\[
    \sum_{\sigma\in S_t} V_H(\bx_{\sigma(1)},\dots,\bx_{\sigma(t)};G) = \partial_{\lambda_1}\dotsm \partial_{\lambda_t} h_H(\lambda_1\bx_1 +\dots+ \lambda_t\bx_t; G),
\]
where \(S_t\) is the symmetric group on \([t]\). In particular, as \(V_{K_t}(\bx_1,\dots,\bx_t;G)\) is symmetric,
\begin{align}\label{eq:G-volume-Kt-pdiff-form}
    V_{K_t}(\bx_1,\dots,\bx_t;G) = \frac{1}{t!} \partial_{\lambda_1}\dotsm \partial_{\lambda_t} h_{K_t}(\lambda_1\bx_1 +\dots+ \lambda_t\bx_t; G).
\end{align}

\paragraph{Lorentzian polynomials.}
Let \(f\) be a polynomial on \(n\) variables \(x_1,\dots,x_n\). The Hessian \(\Hess f\) of $f$ is the \(n\times n\) matrix where \((\Hess f)_{i,j} = \partial_i \partial_j f\) for \(1\le i,j\le n\). The \emph{support} \(\supp(f)\) of \(f\) is the set of tuples \((a_1,\dots,a_n)\in \NN_0^n\) such that \(x_1^{a_1}\dotsm x_n^{a_n}\) has a nonzero coefficient in \(f\).

Following Br\"{a}nd\'{e}n and Huh~\cite{branden2020lorentzian}, Lorentzian polynomials $f$ are defined recursively by using partial derivatives $\partial_i f$ and a combinatorial property of the support of $f$, the so-called \emph{\Mc-convexity}.
A set of vectors \(S\subseteq \NN_0^n\) is \emph{\Mc-convex} if the following \emph{exchange property} holds: for any vectors \(\ba=(a_1,\dots,a_n)\) and \(\bb=(b_1,\dots,b_n)\) in \(S\), whenever \(a_i>b_i\) for some \(i\in [n]\), there is \(j\in [n]\) such that \(a_j<b_j\) and \(\ba-\be_i+\be_j\in S\). Here \(\be_i\) and \(\be_j\) are the \(i\)-th and \(j\)-th standard unit vectors in \(\NN_0^n\), respectively.
Roughly speaking, \Mc-convexity of $S$ defines a base of a discrete polymatroid, a multiset analogue of a matroid. We refer the reader to~\cite{herzog2002discrete} for more discussions about polymatroids.

Although there is no harm in saying that all homogeneous linear polynomials with nonnegative coefficients are Lorentzian, it only causes extra technicalities to carry on. Thus, we restrict ourselves to those homogeneous polynomials of degree at least two. 

A homogeneous polynomial $f$ with nonnegative coefficients of degree $d\geq 2$ is said to be \emph{Lorentzian} if: 
\begin{enumerate}
    \item for $d=2$, the Hessian $\Hess f$ is antiferromagnetic;
    \item for $d>2$, each partial derivative $\partial_i f$ is Lorentzian and $\supp(f)$ is \Mc-convex.
\end{enumerate}
We remark that, while not stated explicitly above, a quadratic Lorentzian polynomial \(f\) also has \Mc-convex support.
We refer to~\cite[Theorem~5.3]{choe2004homogeneous} and \cite[Theorem~3.2]{branden2007polynomials} for the proof of this fact.

\begin{proposition}\label{prop:M-convex_support}
   Let $f$ be a quadratic homogeneous polynomial. Then \(\supp(f)\) is \Mc-convex if \(\Hess f\) is antiferromagnetic.
   In particular, the support of an antiferromagnetic graph $G$ is again antiferromagnetic.
\end{proposition}
To see why the `in-particular' part holds, recall that an \(n\)-vertex weighted graph or its adjacency matrix, denoted by \(G\), naturally corresponds to the quadratic $n$-variable polynomial \(Q_G(\bx) \defeq \frac{1}{2} \bx^T G \bx\),
whose Hessian is exactly $G$.
Then \(G\) is antiferromagnetic if and only if \(Q_G\) is Lorentzian. 
Let $G'=\supp(G)$ for an antiferromagnetic graph $G$. Then \(\supp(Q_{G'}) = \supp(Q_G)\) is \Mc-convex, so \(Q_{G'}\) is Lorentzian by \cite[Lemma~3.11]{branden2020lorentzian}, whence \(G'\) is also antiferromagnetic.

For a graph $G$, a \emph{blow-up} of $G$ means an $\HH$-blow-up of $G$ for $\HH$ that consists of independent sets. 
That is, we replace each $v_i\in V(G)$ by its `clones' $v_i^{(j)}$, $j=1,2,\dots,r_i$ for some positive integer $r_i$ and put an edge between $v_i^{(j)}v_k^{(\ell)}$ if and only if $v_iv_k\in E(G)$.
We extend this definition to weighted graphs $G$ possibly with loops by putting the weight $G(v_i,v_k)$ to each pair $v_i^{(j)}v_k^{(j)}$.
In particular, if $v_i$ is a looped vertex, its clones \(v_i^{(1)},\dots,v_i^{(r_i)}\) form a complete graph $K_{r_i}$ plus $r_i$ loops, whereas clones of a non-looped vertex form an independent set. 
The following lemma, which generalises~\Cref{prop:induced}, states that a nonnegative linear change of variables preserves the Lorentzian property. 
It is a direct consequence of \cite[Theorem~2.10]{branden2020lorentzian} and its proof therein, so we omit the proof.

\begin{lemma}\label{lem:G-vol-Lor-preserved-by-G-blowup-indsubgp}
    Let \(G'\) be a blow-up or an induced subgraph of \(G\). Then the following holds:
    \begin{enumerate}
        \item if \(h_H(\blank;G)\) has \Mc-convex support, then so does \(h_H(\blank;G')\);
        \label{case:G-vol-Lor-preserved-by-G-blowup-indsubgp_M-conv}
        \item if \(h_H(\blank;G)\) is Lorentzian, then so is \(h_H(\blank;G')\).
        \label{case:G-vol-Lor-preserved-by-G-blowup-indsubgp_Lor}
    \end{enumerate}
\end{lemma}

\section{Lorentzian \texorpdfstring{$G$}{G}-chromatic functions}\label{sec:G_chromatic}

Our first step to prove~\Cref{thm:main} is to show that $h_{K_t}(\bx;G)$ is Lorentzian whenever $G$ is antiferromagnetic.
\begin{theorem}\label{thm:AFM-hom-Lor}
    Let \(G\) be an antiferromagnetic graph and let \(t\ge2\). Then \(h_{K_t}(\bx;G)\) is Lorentzian.
\end{theorem}

By the recursive definition of Lorentzian polynomials, there are two facts to check: first, $\partial_i h_{K_t}(\bx;G)$ is Lorentzian for every $i=1,\dots,n$, where $n=v(G)$; second, the support of $h_{K_t}(\bx;G)$ is \Mc-convex.
The first part is rather straightforward to verify. Intuitively, the partial derivative $\partial_i h_{K_t}(\bx;G)$ corresponds to the number of (weighted) copies of $K_t$ that uses the vertex $i\in V(G)$. For unweighted $G$, this explanation is essentially the proof itself, which is presented in our extended abstract for FPSAC~2025~\cite{lee2025counting-FPSAC}. For weighted graphs $G$, the notation might look a little complicated but the conceptual essence of the argument remains the same.

Second, to prove that $\supp(h_{K_t}(\bx;G))$ is \Mc-convex, we give a structural characterisation of antiferromagnetic graphs $G$, which is given in~\Cref{thm:AFM-gph-equiv}. This can be seen as a generalisation of \cite[Corollary~5.4]{choe2004homogeneous}, which states that a loopless graph without isolated vertices is antiferromagnetic if and only if it is a complete multipartite graph. The only difference is that `apex' loops, which connect to all the vertices, may appear.
Let \(K_q^{\circ}\) denote the complete graph on \(q\) vertices with exactly one vertex looped; recall that this corresponds to a semiproper colouring.

\begin{theorem}\label{thm:AFM-gph-equiv}
    Let $G$ be a (unweighted) graph without isolated vertices. Then the following are equivalent:
    \begin{enumerate}
        \item\label{cond:(thm:AFM-gph-equiv)-G-AFM} \(G\) is antiferromagnetic;
        \item\label{cond:(thm:AFM-gph-equiv)-PG-Lor} \(Q_G\) is Lorentzian;
        \item\label{cond:(thm:AFM-gph-equiv)-E(G)-Mconv} \(\supp(Q_G)\) is \Mc-convex;
        \item\label{cond:(thm:AFM-gph-equiv)-G-str} there exist disjoint vertex sets \(V_1\) and \(V_2\) such that $V_1\cup V_2=V(G)$, $V_1$ induces a complete multipartite graph, and $V_2$ consists of looped vertices that connect to all the vertices in $G$.

        That is, \(G\) is obtained by blowing up \(K_q\) or \(K_q^\circ\).% \comment{Do we need `connected'?}
    \end{enumerate}
\end{theorem}

\begin{proof}
    It was already mentioned in~\Cref{sec:prelim} that 
    \ref{cond:(thm:AFM-gph-equiv)-G-AFM},
    \ref{cond:(thm:AFM-gph-equiv)-PG-Lor},
    and \ref{cond:(thm:AFM-gph-equiv)-E(G)-Mconv} are equivalent by \cite[Lemma~3.11]{branden2020lorentzian}.
    
    \noindent \(\text{\ref{cond:(thm:AFM-gph-equiv)-G-AFM}}\implies \text{\ref{cond:(thm:AFM-gph-equiv)-G-str}}\). If $G$ is antiferromagnetic, then so are all its induced subgraphs by~\Cref{prop:induced}. Thus, the graphs \twoedges, \paththree, \triangleaddedge, and {\twoloops} are forbidden as an induced subgraph. It follows that non-looped vertices form a complete multipartite graph and looped vertices are pairwise adjacent.

    Suppose \(G\) has \loopplusvtx\, as an induced subgraph. As \(G\) has no isolated vertex, $G$ must have \loopplusedge, \loopedpathtwo, \loopplusloopededge, or {\twoloopspathtwo} as an induced subgraph. However, none of these are antiferromagnetic, which contradicts antiferromagnetism of $G$. Thus, each looped vertex connects to all the non-looped vertices. Together with the above paragraph, this characterises the structure of $G$ as given in \ref{cond:(thm:AFM-gph-equiv)-G-str}.
    
    \noindent \(\text{\ref{cond:(thm:AFM-gph-equiv)-G-str}}\implies\text{\ref{cond:(thm:AFM-gph-equiv)-G-AFM}}\). By direct calculation, $K_q^\circ$ is antiferromagnetic. By~\Cref{lem:G-vol-Lor-preserved-by-G-blowup-indsubgp}, so is~$G$.
\end{proof}

\begin{proof}[Proof of {\Cref{thm:AFM-hom-Lor}}]
    Write \(V(G)=[n]\) and \(\bx=(x_1,\dots,x_n)\). We use induction on \(t\), where the base case follows from \(h_{K_2}(\bx;G) = Q_G(\bx)\).
    Firstly, we claim that \(\partial_\nu h_{K_t}(\bx;G)\) is Lorentzian for all \(1\le\nu\le n\). To see this, note that %We have, with \(\one(\blank)\) being the true-false indicator of a sentence,
    \begin{align*}
        \partial_\nu h_{K_t}(\bx;G)
        &= \sum_{\phi\colon [t]\to [n]}
            \prod_{\{i,j\}\in \binom{[t]}{2}} G(\phi(i),\phi(j))
            \biggl( \frac{\partial}{\partial x_\nu}
                \prod_{s\in [t]} x_{\phi(s)} \biggr)
        \\&= \sum_{\phi\colon [t]\to [n]}
            \prod_{\{i,j\}\in \binom{[t]}{2}} G(\phi(i),\phi(j))
            \sum_{k\in [t]} \biggl( \one(\phi(k)=\nu)
            \prod_{s\in [t]\setminus\{k\}} x_{\phi(s)} \biggr).
    \end{align*}
    Rearranging the order of the sum by fixing $k\in [t]$ with $\phi(k)=\nu$ gives
    \[
        \partial_\nu h_{K_t}(\bx;G)
        = \sum_{k\in [t]}
            \,\sum_{\phi \colon[t]\to [n]}
            \one(\phi(k)=\nu)
            \prod_{\{i,j\}\in \binom{[t]}{2}} G(\phi(i),\phi(j))
            \prod_{s\in [t]\setminus \{k\}} x_{\phi(s)}.
    \]
    By symmetry of $K_t$, fixing $\phi(k)=\nu$ and summing $\prod_{\{i,j\}\in \binom{[t]}{2}} G(\phi(i),\phi(j))\prod_{s\in [t]\setminus \{k\}} x_{\phi(s)}$ for all such $\phi$ does not depend on $k$.
    For the sake of simplicity, letting $k=t$ gives
    %The summand of \(\sum_{k\in [t]}\) does not depend on \(k\), and when \(k=t\), the value is
    \begin{align*}
        & \sum_{\phi \colon[t]\to [n]}
            \one(\phi(t)=\nu)
            \prod_{\{i,j\}\in \binom{[t]}{2}} G(\phi(i),\phi(j))
            \prod_{s\in [t-1]} x_{\phi(s)}
        \\ ={} & \sum_{\psi\in [t-1]\to [n]}
            \prod_{\{i,j\}\in \binom{[t-1]}{2}} G(\psi(i),\psi(j))
            \prod_{s\in [t-1]} G(\psi(s),\nu)\, x_{\psi(s)}
        \\ ={} & h_{K_{t-1}}(y_1,\dots,y_n;G),
    \end{align*}
    where \(y_i\defeq G(i,\nu)\, x_i\) for \(1\le i\le n\). Therefore,
    \[
        \partial_\nu h_{K_t}(\bx;G) = t\cdot h_{K_{t-1}}(y_1,\dots,y_n;G).
    \]
    By \cite[Theorem~2.10]{branden2020lorentzian}, Lorentzian polynomials remain Lorentzian under affine transformations, i.e.,  $(x_1,\dots,x_n)\mapsto (a_1x_1,\dots,a_nx_n)$ for $a_i\geq 0$, $i=1,2,\dots,n$.
    Together with the induction hypothesis, this fact implies that \(h_{K_{t-1}}(y_1,\dots,y_n;G)\) is Lorentzian, whence \(\partial_\nu h_{K_t}(\bx;G)\) is also.

   Next, we claim that \(\supp(h_{K_t}(\bx;G))\) is \Mc-convex. Let $G_0\defeq\supp(G)$.
   Then 
    \begin{align*}
        \supp(h_{K_t}(\bx;G)) = \supp(h_{K_t}(\bx;G_0)),
    \end{align*}
    as changing all the positive edge weights of $G$ to $1$ turns $h_{K_t}(\bx;G)$ into $h_{K_t}(\bx;G_0)$. 
   By~\Cref{prop:M-convex_support}, \(G_0\) is also antiferromagnetic. \Cref{thm:AFM-gph-equiv}~\ref{cond:(thm:AFM-gph-equiv)-G-str} then implies that \(G_0\) is obtained from \(K_q^\circ\) for some \(q\ge1\) by blowing-up and taking an induced subgraph.
   
    It then suffices to show that \(\supp(h_{K_t}(\bx;K_q^\circ))\) is \Mc-convex, since \Mc-covexity of \(\supp(h_{K_t}(\bx;K_q^\circ))\) implies \Mc-convexity of $\supp(h_{K_t}(\bx;G_0))$ by~\Cref{lem:G-vol-Lor-preserved-by-G-blowup-indsubgp}. Write \(V(K_q^\circ) = \{w_1,\dots,w_q\}\) where \(w_q\) is the looped vertex. Let \(\ba=(a_1,\dots,a_q)\) and \(\bb=(b_1,\dots,b_q)\) be in \(\supp(h_{K_t}(\bx;K_q^\circ))\), and \(a_i>b_i\) for some \(i\in [q]\). Note \(a_j,b_j\in\{0,1\}\) for \(j\in [q-1]\). There are three cases: if \(a_q<b_q\), then \(i\neq q\) and \(\ba-\be_i+\be_q\in \supp(h_{K_t}(\bx;K_q^\circ))\). If \(a_q=b_q\), then \(i\in [q-1]\) and \(a_1+\dots+a_{q-1} = b_1+\dots+b_{q-1}\), so there is \(j\in [q-1]\setminus\{i\}\) such that \(a_j=0<1=b_j\). If \(a_q>b_q\), then \(a_1+\dots+a_{q-1} < b_1+\dots+b_{q-1}\), so again there is such \(j\in [q-1]\). In the two latter cases, \(\ba-\be_i+\be_j\in \supp(h_{K_t}(\bx;K_q^\circ))\). Therefore, \(\supp(h_{K_t}(\bx;K_q^\circ))\) is \Mc-convex.
\end{proof}

Having seen the reduction in the last few paragraphs of the proof, one may wonder whether \Mc-convexity of $h_H(\bx;G)$ always reduces to the \Mc-convexity of $h_H(\bx;K_q^\circ)$ or even $h_H(\bx;K_q)$ for some $q$. We show that this is indeed the case.

\begin{proposition}\label{prop:reduction}
    Let \(H\) be a connected graph on \(t\ge2\) vertices. If \(h_H(\blank;K_q)\) has \Mc-convex support for some \(q\ge t\), then so does \(h_H(\blank;G)\) for all antiferromagnetic graphs \(G\).
\end{proposition}
\begin{proof}
    Let \(G\) be an $n$-vertex antiferromagnetic graph.
    By~\Cref{prop:M-convex_support}, \(\supp(G)\) is an unweighted antiferromagnetic graph. whence \Cref{thm:AFM-gph-equiv}~\ref{cond:(thm:AFM-gph-equiv)-G-str} gives that $\supp(G)$ is an induced subgraph of a blow-up of $K_m^\circ$. Thus, \Cref{lem:G-vol-Lor-preserved-by-G-blowup-indsubgp}~\ref{case:G-vol-Lor-preserved-by-G-blowup-indsubgp_M-conv} allows us to assume that $G=K_m^\circ$ for some $m$.

    We identify each homomorphism from $H$ to $K_{m}^\circ$ as a \emph{semiproper} $m$-colouring, where the colour $m$ can be used with no constraints while no adjacent pairs in $H$ get monochromatic with the other colours.
    Let $c_1$ and $c_2$ be two semiproper $m$-colourings of $H$. The \Mc-convexity then follows if we can `recolour' $c_1$ by changing the colour $r_1=c_1(v)$ of a vertex $v$ into $r_2$, where $r_1$ is more popular in $c_1$ than in $c_2$ and $r_2$ is more popular in $c_2$ than in $c_1$.
    For brevity, we say that this alteration of $c_1$ is a \emph{recolouring of $c_1$ towards  $c_2$}.

    If $c_2$ uses a colour \(i\) that never appears in $c_1$, then the recolouring of $c_1$ towards $c_2$ is constructed by using \(i\). In addition, a colour \(i\) with \(c_1^{-1}(i)=c_2^{-1}(i)=\emptyset\) does not affect the recolouring of \(c_1\) towards \(c_2\). Thus, we may ignore those colours which are not used by \(c_1\) and assume \(c_1^{-1}(i)\neq\emptyset\) for all \(i\in [m]\). Moreover, the `looped' colour $m$ must be used in $c_1$ as much as in $c_2$, as otherwise, the colour $m$ provides an option to recolour $c_1$. Let $k\defeq \abs{c_1^{-1}(m)}$.

    By~\Cref{lem:G-vol-Lor-preserved-by-G-blowup-indsubgp}~\ref{case:G-vol-Lor-preserved-by-G-blowup-indsubgp_M-conv}, \(\supp(h_H(\bx;K_t))\) is \Mc-convex. Turn $c_1$ into a (proper) $t$-colouring $c_1'$ by converting \(m\) into a non-looped colour and replacing each previous use of $m$ by a new distinct colour in $\{m,m+1,\dots,m+k-1\}$. Note this is possible as \(m-1\le \abs{c_1^{-1}(1)} + \dots + \abs{c_1^{-1}(m-1)} = t-k\) so that \(m+k-1\le t\). Using the same set of new colours \(\{m,m+1,\dots,m+k-1\}\), it is possible to make $c_2$ into a $t$-colouring $c_2'$ by assigning the new distinct colours to those vertices which was previously coloured by \(m\), as $\abs{c_2^{-1}(m)}\le k$.
    Then by \Mc-convexity of \(\supp(h_H(\bx;K_t))\), there is a recolouring of $c_1'$ towards $c_2'$.

    The recoloured vertex in $H$ by this recolouring does not use the new colours in \(\{m,m+1,\dots,m+k-1\}\), as these cannot be more popular in $c_2'$ than in $c_1'$. Thus, turning the new colours into the colour $m$ gives a semiproper colouring, which is a recolouring of $c_1$ towards $c_2$.
\end{proof}

This result may be potentially useful in finding other graphs $H$ than the complete graphs that give a Lorentzian $G$-chromatic function $h_H(\bx;G)$ for some particular antiferromagnetic graphs $G$.
We shall discuss more along these lines and propose open problems in the concluding remarks.

\section{Alexandrov--Fenchel-type inequalities for \texorpdfstring{$G$}{G}-volumes}
\label{sec:AF-ineq_for_G-vol}

Proposition~4.5 in \cite{branden2020lorentzian} states that one can obtain an Alexandrov--Fenchel-type inequality from any Lorentzian polynomial.

\begin{proposition}[Proposition~4.5 in {\cite{branden2020lorentzian}}]
\label{prop:Lor-AF-ineq}
    Let \(f\) be a homogeneous polynomial of degree \(d\) in \(n\) variables. Let \(F_f\colon (\RR^n)^d\to \RR\) be defined by
    \[
        F_f(v_1,\dots,v_d) \defeq \frac{1}{d!} \partial_{x_1}\dotsm \partial_{x_d} f(x_1v_1+\dots+x_dv_d).
    \]
    If \(f\) is Lorentzian, then for any \(v_1\in \RR^n\) and \(v_2,\dots,v_d\in (\RR_{\ge0})^n\),
    \[
        F_f(v_1,v_2,v_3,\dots,v_d)^2 \ge F_f(v_1,v_1,v_3,\dots,v_d) \cdot F_f(v_2,v_2,v_3,\dots,v_d).
    \]
\end{proposition}

For example, $f$ can be the volume of the Minkowski sum $x_1K_1+\cdots+x_dK_d$ of convex bodies $K_1,\dots, K_d$, a Lorentzian polynomial given in~\cite[Theorem~4.1]{branden2020lorentzian}. Then~\Cref{prop:Lor-AF-ineq} recovers the classical Alexandrov--Fenchel inequality for mixed volumes.

If $f=h_{K_t}(\bx;G)$ is the $G$-chromatic function, then $F_f$ corresponds to the $G$-volume $V_{K_t}(\blank;G)$.
By~\Cref{thm:AFM-hom-Lor}, $f$ is Lorentzian, so~\Cref{prop:Lor-AF-ineq} yields an Alexandrov--Fenchel-type inequality for the \(G\)-volume of a complete graph.  That is, for \(\ba_1,\ba_2,\dots,\ba_t\in (\RR_{\ge0})^n\),
\begin{equation}\label{eq:AF}
    V_{K_t}(\ba_1,\ba_2,\ba_3\dots,\ba_t;G)^2
    \ge V_{K_t}(\ba_1,\ba_1,\ba_3,\dots,\ba_t;G)
        \cdot V_{K_t}(\ba_2,\ba_2,\ba_3,\dots,\ba_t;G).
\end{equation}
Applying~\eqref{eq:AF} iteratively gives the following inequality.
\begin{corollary}\label{cor:G-vol}
    Let \(G\) be an $n$-vertex antiferromagnetic graph and let \(\ba,\bb\in (\RR_{\ge0})^n\).  Then
    \[
         V_{K_t}(\bb,\ba,\dots,\ba;G)\cdot V_{K_t}(\ba,\bb,\dots,\bb;G)
        \ge 
        V_{K_t}(\ba,\ba,\dots,\ba;G)\cdot V_{K_t}(\bb,\bb,\dots,\bb;G).
    \]
\end{corollary}

We are now ready to prove our main result of this section,
which implies~\Cref{thm:main} by setting \(\ba=\bb=(1,1,\dots,1)\).
In what follows, we enumerate the vertices of $K_t\times K_2$ in the lexicographic order $(1,1),\dots,(t,1),(1,2),\dots,(t,2)$, where \(([t]\times\{1\})\sqcup ([t]\times\{2\})\) is the bipartition and $(i,1)$ and $(i,2)$ are not adjacent for $1\leq i\leq t$.

\begin{theorem}\label{thm:G-vol-ineq-KtxK2}
    Let \(G\) be an $n$-vertex antiferromagnetic graph and let \(t\ge2\). Then for \(\ba,\bb\in (\RR_{\ge0})^n\),
    \begin{equation}\label{eq:G-vol-ineq-KtxK2}
        V_{K_t}(\ba,\dots,\ba; G) \cdot V_{K_t}(\bb,\dots,\bb; G) \le V_{K_t\times K_2}(\underbrace{\ba,\dots,\ba}_t,\underbrace{\bb,\dots,\bb}_t; G).
    \end{equation}
\end{theorem}

The proof proceeds by induction on \(t\). As a warm-up, it would be helpful to see the proof of the following proposition, which uses the inductive technique in a simpler setting.

\begin{proposition}\label{prop:G-vol-ineq-KtxK2_idea}
    Let \(G\) be an unweighted loopless graph and \(t\ge2\). Suppose \eqref{eq:G-vol-ineq-KtxK2} holds for \(t\) replaced by \(t-1\) and all \(\ba,\bb\in \{0,1\}^n\). Then \eqref{eq:G-vol-ineq-KtxK2} holds for \(t\) and \(\ba=\bb=(1,\dots,1)\), i.e.,
    \begin{align}\label{eq:bipartite-swapping-Kt}
       \hom(K_t,G)^2 \le \hom(K_t\times K_2,G).
    \end{align}
\end{proposition}
\begin{proof}
    The proof relies on two simple facts: \(K_t\) minus a vertex is \(K_{t-1}\), and \(K_t\times K_2\) minus two vertices \((t,1)\) and \((t,2)\) is \(K_{t-1}\times K_2\).
    For \(r\in G\), let \(N(r)\) be its set of neighbours. Taking a homomorphic copy of \(K_t\times K_2\) in \(G\) can be done in two steps. First, choose vertices \(r,s\in G\) to embed \((t,1)\) and \((t,2)\), respectively. Next, take a homomorphic copy of \(K_{t-1}\times K_2\) so that the vertices in \([t-1]\times\{1\}\) and \([t-1]\times\{2\}\) are mapped into \(N(s)\) and \(N(r)\), respectively. Applying the assumption gives
    \begin{align*}
        V_{K_t\times K_2}(\one,\dots,\one; G)
        &= \sum_{r,s\in V(G)} V_{K_{t-1}\times K_2}(\underbrace{\one_{N(s)},\dots,\one_{N(s)}}_{t-1},
            \underbrace{\one_{N(r)},\dots,\one_{N(r)}}_{t-1} ; G)
        \\&\ge \sum_{r,s\in V(G)} V_{K_{t-1}}(\one_{N(s)},\dots,\one_{N(s)} ; G)
            \cdot V_{K_{t-1}}(\one_{N(r)},\dots,\one_{N(r)} ; G)
        \\&= \sum_{r\in V(G)} V_{K_{t-1}}(\one_{N(r)},\dots,\one_{N(r)} ; G)
            \sum_{s\in V(G)} V_{K_{t-1}}(\one_{N(s)},\dots,\one_{N(s)} ; G).
    \end{align*}
    The term \(V_{K_{t-1}}(\one_{N(r)},\dots,\one_{N(r)})\) counts the number of the (labelled) \(K_{t-1}\) in the neighbour of \(r\). Summing this over all \(r\in V(G)\) gives the number of the labelled \(K_t\) in \(G\), i.e.,
    \[
        \sum_{r\in V(G)} V_{K_{t-1}}(\one_{N(r)},\dots,\one_{N(r)} ; G)
        = V_{K_t}(\one,\dots,\one; G).
    \]
    Hence, $V_{K_t\times K_2}(\one,\dots,\one; G)\geq V_{K_t}(\one,\dots,\one; G)^2$, as desired.
\end{proof}

The proof of~\Cref{thm:G-vol-ineq-KtxK2} uses the same idea, though `updating' the vertex weights $\ba$ and $\bb$ complicates the notations.

\begin{proof}[Proof of~\Cref{thm:G-vol-ineq-KtxK2}]
    Write
    \[
         V_{K_t\times K_2}(\ba;\bb;G) \defeq V_{K_t\times K_2}(\underbrace{\ba,\dots,\ba}_t, \underbrace{\bb,\dots,\bb}_t;G).
    \]
    
    We use induction on \(t\). By \Cref{cor:G-vol} and the fact that \(K_2\times K_2 \cong K_2\sqcup K_2\), we have
    \[
        V_{K_2}(\ba,\ba;G)\cdot V_{K_2}(\bb,\bb;G)\le V_{K_2}(\ba,\bb;G)^2 = V_{K_2\times K_2}(\ba,\ba,\bb,\bb;G),
    \]
    which verifies the base case \(t=2\).
    
    Suppose \(t\ge 3\). Let \(V(G)=[n]\), \(\ba=(a_1,\dots,a_n)\), and \(\bb=(b_1,\dots,b_n)\).
    \begin{figure}[!htb]
        \centering
        \includegraphics{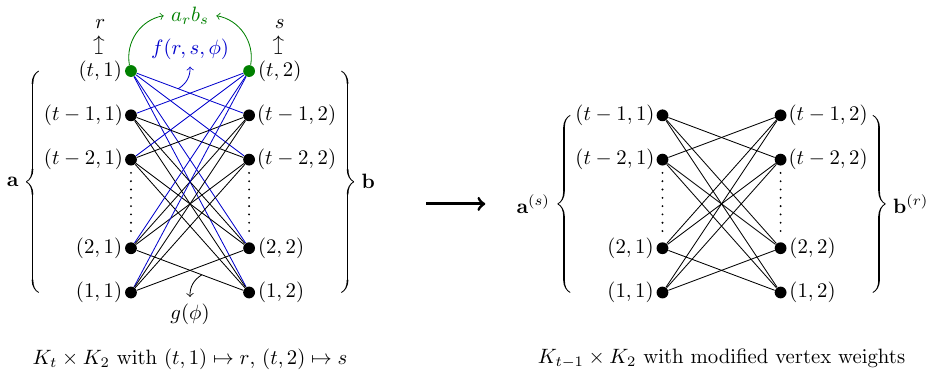}
        \caption{Illustration of the embedding strategy}
        \label{fig:(thm:G-vol-ineq-KtxK2)-pf_description}
    \end{figure}
    We compute the $G$-volume $V_{K_t\times K_2}(\ba;\bb;G)$ by counting (weighted) homomorphisms \(\phi\) from \(K_t\times K_2\) to \(G\) recursively as follows: first, fix \(r=\phi(t,1)\) and \(s=\phi(t,2)\); next, find a (weighted) homomorphic copy of \(K_{t-1}\times K_2\) such that the vertices in \([t-1]\times\{1\}\) and \([t-1]\times\{2\}\) are adjacent to \(s\) and \(r\), respectively.
    See~\Cref{fig:(thm:G-vol-ineq-KtxK2)-pf_description} for an illustration of this embedding. To formalise this two-step embedding process,
    we write
    \begin{equation}\label{eq:G-volume-KtK2}
        V_{K_t\times K_2}(\ba;\bb; G)
        = \sum_{r,s\in V(G)} a_rb_s
            \sum_{\phi\in \Hom(K_{t-1}\times K_2, G)} f(r,s,\phi)\, g(\phi),
    \end{equation}
    where
    \begin{align*}
        f(r,s,\phi) &= \prod_{i\in [t-1]} G(\phi(i,1),s)\, G(r,\phi(i,2)) \quad\text{and} \\
        g(\phi)
        &=  \prod_{i\in [t-1]} a_{\phi(i,1)}
            \prod_{j\in [t-1]} b_{\phi(j,2)}
            \prod_{i,j\in [k], i\neq j} G(\phi(i,1),\phi(j,2)).
    \end{align*}
    Here $f(r,s,\phi)$ encodes the edge weights coming from the fixed vertices $r$ and $s$ and $g(\phi)$ is the weight of the homomorphism $\phi$ from $K_{t-1}\times K_2$.
    We claim that each $a_rb_s\sum_\phi f(r,s,\phi)g(\phi)$ can be rewritten as a \(G\)-volume of \(K_{t-1}\times K_2\) with modified vertex weights, which fully formalises our embedding strategy.

    To see how the weight modification works, for \(r,s\in V(G)\), let \(\ba^{(s)} = (\alpha_\nu^{(s)} a_\nu:1\le \nu\le n)\) and \(\bb^{(r)} = (\beta_\nu^{(r)} b_\nu:1\le \nu\le n)\) be defined by
    \[
        \alpha_\nu^{(s)}\defeq b_s^{1/(t-1)} \, G(\nu,s)
        \quad\text{and}\quad
        \beta_\nu^{(r)}\defeq a_r^{1/(t-1)} \, G(r,\nu).
    \]
    Then
    \[
        \prod_{i\in [t-1]} \alpha_{\phi(i,1)}^{(s)}
        = b_s\, \prod_{i\in [t-1]} G(\phi(i,1),s)
        \quad\text{and}\quad
        \prod_{i\in [t-1]} \beta_{\phi(i,2)}^{(r)}
        = a_r\, \prod_{i\in [t-1]} G(r,\phi(i,2)),
    \]
    which gives
    \[
        \prod_{i\in [t-1]} \alpha_{\phi(i,1)}^{(s)} \beta_{\phi(i,2)}^{(r)}
        = a_r b_s f(r,s,\phi).
    \]
    Therefore,
    \[
        a_r b_s f(r,s,\phi) g(\phi)
        = \prod_{i\in [t-1]} \bigl( \alpha_{\phi(i,1)}^{(s)} a_{\phi(i,1)} \bigr)
                \bigl( \beta_{\phi(i,2)}^{(r)} b_{\phi(i,2)} \bigr),
    \]
    and hence, 
    \begin{align*}
        V_{K_t\times K_2}(\ba;\bb; G)
        &= \sum_{r,s\in V(G)}\, \sum_{\phi\in \Hom(K_{t-1}\times K_2, G)} a_r b_s f(r,s,\phi) g(\phi)
        \\&= \sum_{r,s\in V(G)} V_{K_{t-1}\times K_2}(\ba^{(s)};\bb^{(r)}; G),
    \end{align*}
    as claimed.
    
    \medskip
    
    By the induction hypothesis,
    \begin{align*}
        V_{K_t\times K_2}(\ba;\bb; G)
        &\ge \sum_{r,s\in V(G)} V_{K_{t-1}}(\ba^{(s)},\dots,\ba^{(s)}; G)
            \cdot V_{K_{t-1}}(\bb^{(r)},\dots,\bb^{(r)}; G)
        \\&= \sum_{s\in V(G)} V_{K_{t-1}}(\ba^{(s)},\dots,\ba^{(s)}; G)
            \cdot \sum_{r\in V(G)} V_{K_{t-1}}(\bb^{(r)},\dots,\bb^{(r)}; G).
    \end{align*}
    Simplifying the first factor in the RHS yields
    \begin{align*}
        &\hphantom{{}={}}\!%
        \sum_{s\in V(G)} V_{K_{t-1}}(\ba^{(s)},\dots,\ba^{(s)}; G)
        \\&= \sum_{s\in V(G)}
                \,\sum_{\phi\in \Hom(K_{t-1},G)}
                \,\prod_{\{i,j\}\in \binom{[t-1]}{2}} G(\phi(i),\phi(j))
                \prod_{i\in [t-1]} \alpha_{\phi(i)}^{(s)} a_{\phi(i)}
        \\&= \sum_{s\in V(G)}
                \,\sum_{\phi\in \Hom(K_{t-1},G)}
                \,\prod_{\{i,j\}\in \binom{[t-1]}{2}} G(\phi(i),\phi(j))
                \prod_{i\in [t-1]} b_s^{1/(t-1)} G(s,\phi(i)) a_{\phi(i)}
        \\&= \sum_{\phi'\in \Hom(K_t,G)}
                \,\prod_{\{i,j\}\in \binom{[t]}{2}} G(\psi(i),\psi(j))
                \,\Biggl( \prod_{i\in [t-1]} a_{\psi(i)} \Biggr)
                b_{\psi(t)}
        \\&= V_{K_t}(\ba,\dots,\ba, \bb ; G),
    \end{align*}
    and symmetrically, it also holds that
    \[
        \sum_{r\in V(G)} V_{K_{t-1}}(\bb^{(r)},\dots,\bb^{(r)}; G)
        = V_{K_t}(\bb,\dots,\bb, \ba ; G).
    \]
    Therefore,
    \begin{equation}\label{eq:G-vol-int-ineq} % intermediate
        V_{K_t\times K_2}(\ba;\bb; G)
        \ge V_{K_t}(\ba,\dots,\ba, \bb ; G)
            \cdot V_{K_t}(\bb,\dots,\bb, \ba ; G),
    \end{equation}
    which, together with \Cref{cor:G-vol}, concludes the proof.
\end{proof}

% ================ BIPARTITE SWAPPING GRAPHS ================
\section{\texorpdfstring{$\HH$}{H}-blow-up of bipartite graphs}\label{sec:BS}
Our remaining goal in what follows is to prove  \Cref{thm:main2,thm:main}, both of which obtain inequalities of the form
\begin{equation}\label{eq:bipartite-swapping}
    \hom(H,G)^2 \le \hom(H\times K_2,G),
\end{equation}
for various graphs $H$ and some antiferromagnetic graphs $G$. In his pioneering work~\cite{zhao2010number,zhao2011bipartite}, Zhao proved the inequality~\eqref{eq:bipartite-swapping} for any graphs $H$ paired with $G=\indepg$ or $G=K_q$, where $q=q(H)$ is large enough, by using his `bipartite swapping trick'.
Analogously, we say that \(H\) is \emph{bipartite swapping in \(G\)} if~\eqref{eq:bipartite-swapping} holds.

Our key idea in proving~\Cref{thm:main2,thm:main} is to extend the techniques in the previous section to replace $K_t$ by other graphs $H$.
Let \(H\) be a $t$-vertex graph.
As done for $H=K_t$ in the previous section, we set \(V(H)=[t]\) and enumerate the vertices of \(H\times K_2\) in the lexicographic order $(1,1),\dots,(t,1),(1,2),\dots,(t,2)$, where \(([t]\times\{1\})\sqcup ([t]\times\{2\})\) is the bipartition coming from the tensor product, and $(i,1)$ and $(j,2)$ are adjacent if and only if \(ij\in E(H)\). We say that \(([t]\times\{1\})\sqcup ([t]\times\{2\})\) is the \emph{canonical bipartition} and $(1,1),\dots,(t,1),(1,2),\dots,(t,2)$ is the \emph{canonical vertex labelling} of \(H\times K_2\) with respect to the vertex labelling of $H$.

Let $H'$ be a bipartite graph with a fixed bipartition $V_1\sqcup V_2$.
For an unweighted graph \(G\) and \(A,B\subseteq V(G)\), let \(\bHom(H', G[A,B])\) denote the set of `bipartite' homomorphisms from \(H'\) to \(G\). That is, those $\phi\in\Hom(H',G)$ which maps $V_1$ to $A$ and $V_2$ to $B$, respectively.\footnote{We abuse the notation $G[A,B]$ slightly, since $G[A,B]$ is not necessarily a bipartite subgraph of $G$, particularly if \(A\) and \(B\) are not disjoint. 
It would be better to think $G[A,B]$ as a new bipartite graph on the bipartition $A'\cup B'$, where $A'$ and $B'$ are disjoint copies of $A$ and $B$, respectively, where $a\in A'$ and $b\in B'$ are adjacent if the corresponding vertices in $G$ are adjacent. Then \(\bhom\) counts the number of bipartite homomorphisms from $H'$ to $G[A,B]$, respecting the order of the bipartitions.}
Writing $\bhom(H', G[A,B])\defeq \abs{\bHom(H', G[A,B])}$ and considering the canonical bipartition of $H\times K_2$, we say that
\(H\) is \emph{cross-bipartite swapping in \(G\)} if for all \(A,B\subseteq V(G)\),
\begin{equation}\label{eq:cross-bipartite-swapping}
    \hom(H,G[A]) \hom(H,G[B])
    \le \bhom(H\times K_2, G[A,B]).
\end{equation}
We remark that every bipartite graph is bipartite swapping in any graph \(G\), but it is not necessarily cross-bipartite swapping.
For instance, if \(E(H)\) is nonempty and \(A\) and \(B\) are disjoint subsets of \(V(G)\) where \(G[A]\) and \(G[B]\) contain homomorphic copies of \(H\) and there are no edges between \(A\) and \(B\), then \(\bhom(H\times K_2, G[A,B])=\bhom(H,G[A,B])^2 =0\) while \(\hom(H,G[A]) \hom(H,G[B]) > 0\).

The inequality~\eqref{eq:cross-bipartite-swapping} can be paraphrased as
\[
    V_{H}(\one_A,\dots,\one_A; G) \cdot V_{H}(\one_B,\dots,\one_B; G)
    \le V_{H\times K_2}(\underbrace{\one_A,\dots,\one_A}_t,\underbrace{\one_B,\dots,\one_B}_t; G),
\]
which is exactly \eqref{eq:G-vol-ineq-KtxK2} with \(\ba\defeq \one_A\) and \(\bb\defeq \one_B\).
Then \Cref{prop:G-vol-ineq-KtxK2_idea} shows that, if \(K_{t-1}\) is cross-bipartite swapping in $G$, then $K_t$ is bipartite swapping in \(G\) provided that \(G\) is unweighted and loopless.
Thus, in a sense,~\eqref{eq:cross-bipartite-swapping} represents a weaker  log-concavity than the Lorentzian property of $K_t$ shown in the previous sections.

The proof of~\Cref{prop:G-vol-ineq-KtxK2_idea} relies on the fact that \(K_t\) consists of \(K_{t-1}\) and a vertex connected to all the vertices in \(K_{t-1}\). This relation between \(K_{t-1}\) and \(K_t\) motivates us to define \(\HH\)-blow-ups. For convenience, we recall that, for a class $\HH$ of graph containing \(K_1\) and a graph \(F\) on the vertex set \(\{v_1,\dots,v_k\}\), an \emph{\(\HH\)-blow-up of \(F\)} is a graph obtained by replacing each \(v_i\) by some \(H_i\in\HH\) and placing a complete bipartite graph between \(V(H_i)\) and \(V(H_j)\) whenever \(v_iv_j\in E(F)\).
\Cref{fig:H-blow-up} describes two examples. The first one is when \(F\) is the single edge graph, whose vertices are blown-up by \(H_1=K_1\) and \(H_2=K_4\), respectively, so that the \(\HH\)-blow-up gives \(K_5\). In general, \(K_t\) is a \(\{K_1,K_{t-1}\}\)-blow-up of an edge. The second example shows the $\HH$-blow-up of a 2-edge path \(F\) with a 2-edge path \(H_1\), \(H_2=K_2\), and \(H_3=K_3\).

\begin{figure}[!htb]
    \centering
    \includegraphics{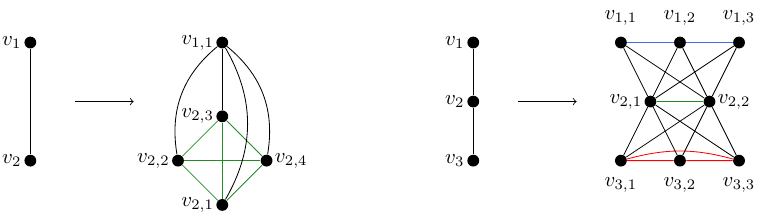}
    \caption{Examples of \(\HH\)-blow-ups}
    \label{fig:H-blow-up}
\end{figure}

The main theorem of this section for unweighted $G$ generalises~\Cref{prop:G-vol-ineq-KtxK2_idea} by using this terminology. 

\begin{theorem}\label{thm:BS-H-blow-up}
    Let \(G\) be a graph. Let \(F\) be a bipartite graph and \(\HH\) be the class of cross-bipartite swapping graphs in \(G\). Then every \(\HH\)-blow-up of \(F\) is bipartite swapping in \(G\).
\end{theorem}

In the subsequent section, we prove that  paths,  even cycles, and complete multipartite graphs are cross-bipartite swapping in every complete graph, which together with~\Cref{thm:BS-H-blow-up} results in \Cref{thm:main2}.

Throughout this section, let \(F\) be a graph on the vertex set \(\{v_1,\dots,v_m\}\) and let \(\HH\) be a graph class containing \(K_1\).
The graph \(H\) denotes an \(\HH\)-blow-up of \(F\) obtained by replacing each \(v_i\) by \(H_i\in\HH\).
For \(1\le i\le m\), let \(V_i\subseteq V(H)\) be the vertex set of \(H_i\) so that \(V(H) = \bigsqcup_{i=1}^m V_i\).

For each \(U\subseteq V(F)\), we define the graph \(\swapped{H}{U}\), whose examples are given in \Cref{fig:2-lifts}, as follows. We start by taking two vertex-disjoint copies of \(H\), each of which is induced on \(\bigsqcup_{i=1}^m W_i\) and \(\bigsqcup_{i=1}^m W'_i\), respectively, where \(W_i\) and \(W'_i\) are copies of \(V_i\), $i=1,2,\dots,m$.
The graph $\swapped{H}{U}$ is then obtained by replacing the copy of \(H_j\sqcup H_j\) induced on \(W_j\sqcup W'_j\) by \(H_j\times K_2\) for each $j$ such that \(v_j\in U\).

\begin{figure}[!htb]
    \centering
    \includegraphics{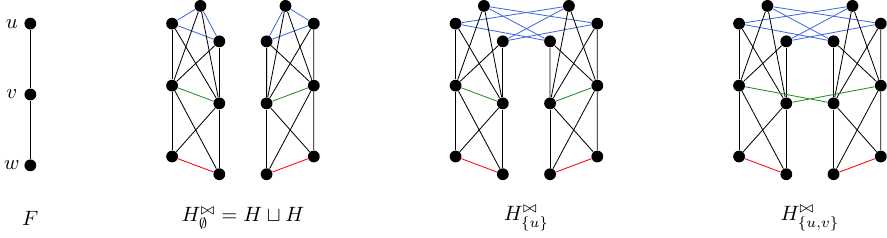}
    \caption{Examples of \(\swapped{H}{U}\)'s}
    \label{fig:2-lifts}
\end{figure}

We show that, if \(\HH\) consists of the cross-bipartite swapping graphs, then the homomorphism counts of the \(\swapped{H}{U}\)'s in \(G\) get larger as we extend $U$. 

\begin{lemma}\label{lem:2-lifts_hom-comparison}
    Let \(G\) be a graph and let \(F\) be a bipartite graph. 
    If \(H\) be an \(\HH\)-blow-up of \(F\), where \(\HH\) consists of cross-bipartite swapping graphs in \(G\), then for \(U\subseteq V(F)\) and \(u\in V(F)\setminus U\),
    \[
        \hom(\swapped{H}{U},G) \le \hom(\swapped{H}{U\cup\{u\}},G).
    \]
\end{lemma}
\begin{proof}
    The idea is to generalise the proof of \Cref{prop:G-vol-ineq-KtxK2_idea}.
    Let \(u=v_1\) by relabelling $V(F)$ if necessary. Then the two graphs \(\swapped{H}{U}\) and \(\swapped{H}{U\cup\{v_1\}}\) on the same vertex set \(\bigsqcup_{i=1}^m (W_i\sqcup W'_i)\) have the same set of edges except those contained in \(W_1\sqcup W'_1\).
    In particular, the subgraph $R$ of $\swapped{H}{U}$ induced on \(\bigsqcup_{i>1} (W_i\sqcup W'_i)\) is also an induced subgraph of $\swapped{H}{U\cup\{v_1\}}$ on the same vertex set.
    
    Let \(W\defeq \bigsqcup_{v_i\in N_F(v_1)} W_i\) and \(W'\defeq \bigsqcup_{v_i\in N_F(v_1)} W'_i\). Denote by \(N(A)\) the set of common neighbours of \(A\subseteq V(G)\).
    Then fixing a homomorphism $\phi\in\Hom(R,G)$ and extending it by finding two homomorphic copies of $H_1$, one in $G[N(\phi(W))]$ and the other in $G[N(\phi(W'))]$, gives an $\swapped{H}{U}$-homomorphism. Conversely, any $\swapped{H}{U}$-homomorphism in $G$ can be obtained by such a way. Analogously, extending $\phi\in\Hom(R,G)$ by finding a homomorphic copy of $H_1\times K_2$, each side of which is in $G[N(\phi(W)]$ and $G[N(\phi(W')]$, gives an $\swapped{H}{U\cup\{v_1\}}$-homomorphism and this procedure characterises all $\swapped{H}{U\cup\{v_1\}}$-homomorphisms.
    Therefore, 
    \begin{align*}
        \hom(\swapped{H}{U},G)
        &= \sum_{\phi\in \Hom(R,G)} \hom\bigl( H_1, G[N(\phi(W)] \bigr)
            \cdot \hom\bigl( H_1, G[N(\phi(W'))] \bigr)
        \\&\le \sum_{\phi\in \Hom(R,G)} \bhom\bigl( H_1\times K_2, G[N(\phi(W)), N(\phi(W'))] \bigr)
        \\&= \hom(\swapped{H}{U\cup \{v_1\}}, G),
    \end{align*}
    where the inequality follows from \(H_1\) being cross-bipartite swapping in \(G\).
\end{proof}

The following lemma, which may be a standard fact for those graphs of the form $H\times K_2$, is the reason why we need bipartiteness of $F$. We include the proof for completeness. 

\begin{lemma}\label{lem:2-lift-iso}
    For a bipartite graph \(F\) and a nonempty graph class \(\HH\), let \(H\) be an \(\HH\)-blow-up of \(F\). Then
    \[
        H\times K_2 \cong \swapped{H}{V(F)}.
    \]
\end{lemma}
\begin{proof}
    \Cref{fig:2-lifts_iso}, which shows examples of \(H\times K_2\) and \(\swapped{H}{V(F)}\), may explain the entire proof; from \(H\times K_2\), switching the positions of the black and the white vertices in the middle layer, which are the copies of \(x_2\) in $H$, gives \(\swapped{H}{V(F)}\). The task is to formalise this isomorphism.

    \begin{figure}[!htb]
        \centering
        \includegraphics{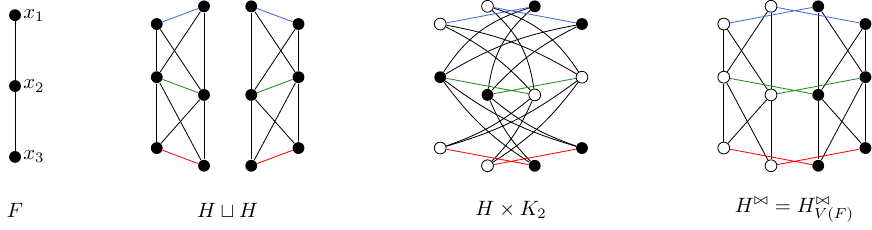}
        \caption{Description of \(H\times K_2 \cong \swapped{H}{V(F)}\)}
        \label{fig:2-lifts_iso}
    \end{figure}
    
    For brevity, write $\swapped{H}{}\defeq\swapped{H}{V(F)}$. Let \(V(F)=A\sqcup B\) be a bipartition of \(F\) (in~\Cref{fig:2-lifts_iso}, \(A=\{x_1,x_3\}\) and \(B=\{x_2\}\)) and enumerate its vertices by $x_1,x_2,\dots,x_m$. 
    Let $V(H) = \bigsqcup_{j=1}^m V_j$, where $V_j$ induces the copy of $H_j$ in $H$ that `blowed up' $x_j\in V(F)$.
    Then $H\times K_2$ is the graph on $W=V(H)\times [2]$ with the canonical bipartition $W_1\sqcup W_2$, i.e., $W_i=\{(v,i):v\in V(H)\}$, $i=1,2$.
    We may assume that the graph $\swapped{H}{}$ is also defined on $W$, where two types of edges exist: first, the edges of the form $\{(v,1),(u,2)\}$, where $uv\in E(H_j)$, i.e., those edges in the copy of $H_j\times K_2$ (coloured in~\Cref{fig:2-lifts_iso}); second, the edges of the form $\{(v,i),(u,i)\}$, $i=1,2$, for $v\in V_j$ and $u\in V_k$, whenever $x_jx_k\in E(F)$, i.e., those `blown-up' edges from $F$-edges.

    Let $\varphi\colon W\rightarrow W$ be the bijection that fixes each $(v,i)$, $i=1,2$, for $v\in V_j$, $x_j\in A$, and swaps $(v,1)$ and $(v,2)$ for $v\in V_j$, $x_j\in B$. As described in~\Cref{fig:2-lifts_iso}, we claim that this is an isomorphism between $H\times K_2$ and $\swapped{H}{}$. Indeed, each edge of the form $\{(v,1),(u,2)\}\in E(H\times K_2)$ with $v\in V_j$, $x_j\in A$, and $u\in V_k$, $x_k\in B$, maps to $\{(v,1),(u,1)\}\in E(\swapped{H}{})$. The other edges of the form $\{(v,1),(u,2)\}\in E(H\times K_2)$ with $u,v\in V_{j}$ such that $uv\in E(H_j)$ remain fixed if $x_j\in A$ or map to $\{(v,2),(u,1)\}\in E(\swapped{H}{})$ otherwise. Thus, $\varphi$ is a bijective homomorphism from $H\times K_2$ to $\swapped{H}{}$. As the two graphs have the same number of edges, $\varphi$ must be an isomorphism.
\end{proof}

\begin{proof}[Proof of \Cref{thm:BS-H-blow-up}]
    Let \(H\) be an \(\HH\)-blow-up of \(F\).
    Applying \Cref{lem:2-lifts_hom-comparison} repeatedly, we have \(\hom(\swapped{H}{\emptyset},G)\le \hom(\swapped{H}{V(F)},G)\). Since \(\hom(\swapped{H}{\emptyset},G) = \hom(H\sqcup H,G) = \hom(H,G)^2\) and \(H\times K_2 \cong \swapped{H}{V(F)}\) by \Cref{lem:2-lift-iso}, the conclusion follows.
\end{proof}

As mentioned already, in the subsequent section, we shall show that complete multipartite graphs, even cycles, and paths are cross-bipartite swapping in $K_q$, $q\geq 2$, to complete the proof of~\Cref{thm:main2}.
The strategy to prove~\Cref{thm:main} is more or less the same. \Cref{thm:G-vol-ineq-KtxK2} shows that complete graphs are cross-bipartite swapping in any antiferromagnetic graphs $G$, even in a stronger `weighted' sense. It is hence enough to obtain an analogue of~\Cref{thm:BS-H-blow-up}. To state this, we say that a $t$-vertex graph \(H\) is \emph{weighted cross-bipartite swapping in a weighted graph \(G\)} if for all \(\ba,\bb\colon V(G)\to \RR_{\ge0}\),
\[% \label{eq:strongly-weighted-bipartite-swapping-V}
     V_{H}(\ba,\dots,\ba; G) \cdot V_{H}(\bb,\dots,\bb; G)
     \le V_{H\times K_2}(\underbrace{\ba,\dots,\ba}_t,\underbrace{\bb,\dots,\bb}_t; G).
\]
A variant of~\Cref{thm:BS-H-blow-up} for weighted graphs $G$ is stated as follows.
\begin{theorem}\label{thm:BS-H-blow-up_weighted}
    Let \(G\) be a weighted graph. Let \(F\) be a bipartite graph and \(\HH\) be the class of weighted cross-bipartite swapping graphs in \(G\). Then every \(\HH\)-blow-up of \(F\) is bipartite swapping in \(G\).
\end{theorem}

\Cref{thm:G-vol-ineq-KtxK2} states that all complete graphs are weighted cross-bipartite swapping in the antiferromagnetic graphs, so \Cref{thm:BS-H-blow-up_weighted} implies \Cref{thm:main}.
The proof of~\Cref{thm:BS-H-blow-up_weighted} is verbatim the same as that of \Cref{thm:BS-H-blow-up} except the part that `updates' vertex weights supported in $N(v)$ using edge weights incident to the vertex $v$, as was done in the proof of \Cref{thm:G-vol-ineq-KtxK2}, so we omit it here.

\section{Cross-bipartite swapping graphs}\label{sec:SBS-in-Kq}
For \(A\subseteq V(K_q)\), \(\hom(H,K_q[A])\) depends only on the size of $A$.
In contrast, $\bhom(H\times K_2, K_q[A,B])$ counts the number of list colourings of $H\times K_2$, where the two parts in $H\times K_2$ receive the colour lists $A$ and $B$, respectively, so it may depend on how $A$ and $B$ intersect. 
Our first lemma states that, to prove the cross-bipartite swapping property~\eqref{eq:cross-bipartite-swapping}, we may assume $A\subseteq B$.

\begin{lemma}\label{lem:SBS-in-Kq_simp-condition}
    A graph \(H\) is cross-bipartite swapping in \(K_q\) if and only if for all nonempty \(A,B\subseteq V(K_q)\) with \(A\subseteq B\),
    \begin{equation}\label{eq:SBS-in-Kq}
        \hom(H,K_q[A]) \hom(H,K_q[B]) \le \bhom(H\times K_2, K_q[A,B]).
    \end{equation}
\end{lemma}
\begin{proof}
    It is enough to show the `if' part, so suppose \eqref{eq:SBS-in-Kq} holds whenever \(A\subseteq B\). 
    We may assume $\abs{A}\leq \abs{B}$ without loss of generality while $A\setminus B$ is nonempty.  
    Let $A'\subseteq B$ be a subset of size $\abs{A}$ containing $A\cap B$ and let \(f\colon A'\setminus A\to A\setminus A'\) be a bijection.

    With the canonical bipartition \(V_1\sqcup V_2\) of \(H\times K_2\), 
    we claim that $\bhom(H\times K_2,K_q[A',B])$ is at most $\bhom(H\times K_2,K_q[A,B])$. If this is true, then
    \begin{align}\label{eq:N(H;a,b)_reduction}
        \hom(H, K_q[A]) \cdot \hom(H,K_q[B])
        &= \hom(H, K_q[A']) \cdot \hom(H,K_q[B]) \nonumber
        \\&\le \bhom(H\times K_2, K_q[A', B]) \nonumber
        \\&\le \bhom(H\times K_2, K_q[A,B]),
    \end{align}
    which completes the proof.
    
    To prove the claim, the idea is to recolour the vertices in \(H\times K_2\) coloured by \(A'\) using the corresponding colours in \(A\) along \(f\).
    Let $\phi$ be the list $q$-colouring of $H$ that colours $V_1$ and $V_2$ with colours in $A'$ and $B$, respectively. Then one can recolour $\phi$ by replacing each use of a colour $a\in A'\setminus A$ by $f(a)\in A\setminus A'$. As $A\setminus A'$ is disjoint from $B$, the new colouring is also a proper $q$-colouring and furthermore, every $\phi$ maps to a unique colouring in $\bHom(H\times K_2,K_q[A,B])$.
\end{proof}

As $\hom(H,K_q[A])=\hom(H,K_{\abs{A}})$, which only depends on $a\defeq\abs{A}$. To stress this fact, we simply write
\[
    N(H;a) \defeq \hom(H,K_a).
\]
By~\Cref{lem:SBS-in-Kq_simp-condition}, 
when counting $\bhom(H',K_q[A,B])$ for a bipartite graph \(H'=H\times K_2\) with a fixed bipartition \(V_1\sqcup V_2\), we may assume $A\subseteq B$ or $B\subseteq A$.
If $A\subseteq B$, then $\bhom(H',K_q[A,B])$ also depends only on $a\defeq\abs{A}$ and $b\defeq\abs{B}$, $a\leq b$, so we write 
\[
    N(H'; a,b) \defeq \bhom(H', K_b[A,B])
\]
for simplicity. 
Symmetrically, we also write $N(H';b,a)\defeq\bhom(H',K_b[B,A])$ for $a\leq b$. If $H'$ admits an automorphism that swaps $V_1$ and $V_2$, then $N(H';b,a)=N(H';a,b)$. We also extend this definition to arbitrary bipartite graphs $H'$ that are not necessarily of the form $H\times K_2$, while specifying the bipartition of $H'$ to avoid ambiguity.

Throughout this section, \(a\) and \(b\) are always positive integers.
By \Cref{lem:SBS-in-Kq_simp-condition}, \(H\) is cross-bipartite swapping in  $K_q$ if and only if for \(1\le a\le b\le q\),
\begin{equation}\label{eq:SBS-goal}
    N(H;a) \, N(H;b)\le N(H\times K_2; a,b).
\end{equation}
When showing \eqref{eq:SBS-goal} for various graphs \(H\), the following elementary observation will be useful.

\begin{proposition}\label{prop:sum-prod-red} Let $(a_n)$, $(b_n)$, $(c_n)$, and $(d_n)$ be real sequences.
    \begin{enumerate}
        \item\label{case:(prop:sum-prod-red)-1} If \(a_k b_\ell \le c_k d_\ell\) for all \(0\le k\le m_1\) and \(0\le \ell\le m_2\), then
        \[
            \sum_{k=0}^{m_1} a_k \sum_{\ell=0}^{m_2} b_\ell \le \sum_{k=0}^{m_1} c_k \sum_{\ell=0}^{m_2} d_\ell.
        \]
        \item\label{case:(prop:sum-prod-red)-2} If \(a_k b_\ell + a_\ell b_k \le c_k d_\ell + c_\ell d_k\) for all \(0\le k\le \ell\le m\), then
        \[
            \sum_{k=0}^{m} a_k \sum_{\ell=0}^{m} b_\ell \le \sum_{k=0}^{m} c_k \sum_{\ell=0}^{m} d_\ell.
        \]
    \end{enumerate}
\end{proposition}

\subsection{Paths and even cycles}
In this subsection, we show the paths and the even cycles are cross-bipartite swapping in $K_q$.
These graphs $H$ are bipartite, so $H\sqcup H\cong H\times K_2$, which reduces \eqref{eq:SBS-goal} to
\begin{equation}\label{eq:SBS-bip-goal}
    N(H;a) \, N(H;b) \le N(H; a,b) \, N(H; b,a).
\end{equation}
Note that if $a=b$, then \eqref{eq:SBS-bip-goal} is an equality, and if \(a=1\), then the LHS vanishes unless \(H\) has no edges. Thus, throughout this subsection, we show \eqref{eq:SBS-bip-goal} for $2\le a < b$. We shall use \(d\) for a positive integer to parametrise lengths of paths and cycles.

For the paths and the even cycles, we count $N(H;a,b)$ and $N(H;b,a)$ by fixing the homomorphic embedding to the $a$-colour side first and then determining the number of choices left on the $b$-colour side. This is done by the following lemma.

\begin{lemma}\label{lem:bip-cross-col}
    Let \(H\) be a bipartite graph on the bipartition \(V_1\sqcup V_2\), where every vertex in \(V_2\) has degree either \(1\) or \(2\). Let \(2\le a\le b\). Let \(M(\ell)\) be the number of the colourings \(\phi\colon V_1\to [a]\) such that there are exactly \(\ell\) vertices \(v\) in \(V_2\) such that \(\abs{\phi(N(v))}=2\). Then
    \[
        N(H;a,b) = \sum_{\ell=0}^{\abs{V_2}}T(a,b,\ell),
    \]
    where $T(a,b,\ell) =  M(\ell) (b-1)^{\abs{V_2}-\ell} (b-2)^\ell$.
\end{lemma}
\begin{proof}
    It is enough to count the possible choices for \(\phi(v)\), \(v\in V_2\), amongst the \(b\) colours avoiding \(\phi(N(v))\). If \(\abs{\phi(N(v))}=1\), there are \(b-1\) choices; otherwise if \(\abs{\phi(N(v))}=2\), \(b-2\) choices exist. Hence, the number of the colourings in this case is \(M(\ell) (b-1)^{\abs{V_2}-\ell} (b-2)^\ell\); summing this up for distinct values of $\ell$ concludes the proof.
\end{proof}

\subsubsection*{Odd–length paths}
\begin{figure}[!htb]
    \centering
    \includegraphics{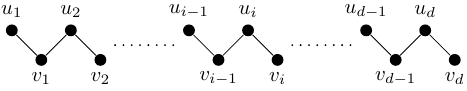}
    \caption{The path \(P\) of odd length \(2d-1\).}
    \label{fig:odd-length-path}
\end{figure}

Let \(P=u_1v_1u_2v_2\cdots u_{d}v_{d}\) be a path of length \(2d-1\) with the bipartition  
\(\{u_{i} : 1\le i\le d\} \sqcup \{v_{j} : 1\le j\le d\}\) (see \Cref{fig:odd-length-path}) and let $2\leq a\le b$.
For a vertex-colouring \(\phi\) counted by \(N(P; a,b)\), let
\[
    L(\phi) \defeq \{2\le i\le d : \phi(u_{i}) \neq \phi(u_{i-1})\}.
\]
This allows us to count those $\phi$ with $\abs{L(\phi)}=\ell$, i.e., $T(a,b,\ell)$ in~\Cref{lem:bip-cross-col}. There are \(\binom{d-1}{\ell}\) ways to choose $L(\phi)$ as an \(\ell\)-subset of \([2,d]\).
After choosing $\phi(u_1)$ amongst $a$ colours, if $i\geq 2$ satisfies $i\notin L(\phi)$, then $\phi(u_i)$ is determined by $\phi(u_{i-1})$. Otherwise, if $i\in L(\phi)$, there are $a-1$ colour choices for $\phi(u_i)$ avoiding $\phi(u_{i-1})$.
Hence, $M(\ell)$ in \Cref{lem:bip-cross-col} is $\binom{d-1}{\ell} a(a-1)^{\ell}$, so
\[
    T(a,b,\ell)=\binom{d-1}{\ell} a(a-1)^{\ell} (b-1)^{d-\ell} (b-2)^\ell
\]
and therefore,
\[
    N(P;a,b)
    =\sum_{\ell=0}^{d-1} T(a,b,\ell)
    =\sum_{\ell=0}^{d-1}
        \binom{d-1}{\ell}
            a(a-1)^{\ell}
            (b-1)^{d-\ell}
            (b-2)^{\ell}.
\]

Our goal~\eqref{eq:SBS-bip-goal} is to show
\[
    N(P; a) \, N(P; b) \le N(P; a,b) \, N(P; b,a) = N(P; a,b)^2
    \qquad \text{for \(2\le a < b\)},
\]
where $N(P;b,a)=N(P;a,b)$ follows from the symmetry of $P$ swapping the two parts of the bipartition.
\Cref{prop:sum-prod-red}~\ref{case:(prop:sum-prod-red)-1} then reduces the problem to show that for all \(0\le k,\ell\le d-1\),
\[
    T(a,a,k) \, T(b,b,\ell) \le T(a,b,k) \, T(a,b,\ell).
\]
That is,
\(
    (a-1)^{d-k-\ell} (a-2)^k b \le  (b-1)^{d-k-\ell} (b-2)^k a
\),
which follows from
\begin{equation}\label{eq:path-SBS-in-Kq_simp}
    \frac{\text{LHS}}{\text{RHS}}
    =\biggl(\frac{a-1}{b-1}\biggr)^{d-k-\ell} \biggl(\frac{a-2}{b-2}\biggr)^k \frac{b}{a}
    \le \biggl(\frac{a-1}{b-1}\biggr)^{d-\ell} \frac{b}{a}
    \le 1.
\end{equation}

\subsubsection*{Even-length paths}
% \paragraph{Even-length paths.}
\begin{figure}[!htb]
    \centering
    \includegraphics{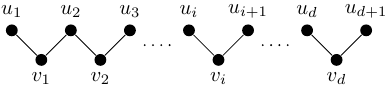}
    \caption{The path \(P\) of even length \(2d\).}
    \label{fig:even-length-path}
\end{figure}

Let \(P=u_1v_1u_2v_2\cdots u_d v_d u_{d+1}\) be the path of length \(2d\) with the bipartition  
\(\{u_i : 1\le i\le d+1\} \sqcup \{v_i : 1\le i\le d\}\) (see \Cref{fig:even-length-path}) and let \(2\le a \le b\). The two parts are no longer symmetric, which makes $N(P;a,b)$ and $N(P;b,a)$ not identical in general and requires some extra lines of calculation.
Again, for a vertex-colouring \(\phi\) counted by \(N(P;a,b)\), let
\[
    L(\phi) \defeq \{2\le i\le d+1 : \phi(u_i)\neq \phi(u_{i-1})\}.
\]
Then the number $T_1(a,b,\ell)$ of those \(\phi\) with \(\abs{L(\phi)}=\ell\) is equal to
\[
    T_1(a,b,\ell)
    =
    \binom{d}{\ell} a(a-1)^\ell (b-1)^{d-\ell} (b-2)^\ell.
\]
We use $T_1(a,b,\ell)$ instead of $T(a,b,\ell)$ to distinguish it from the terms that arise from $N(P;b,a)$ later.
Indeed, there are \(\binom{d}{\ell}\) ways to choose \(L(\phi)\).
After choosing $\phi(u_1)$ amongst $a$ colours, if $i\geq 2$ satisfies $i\notin L(\phi)$, then $\phi(u_i)$ is determined by $\phi(u_{i-1})$. Otherwise, if $i\in L(\phi)$, there are $a-1$ colour choices for $\phi(u_i)$ avoiding $\phi(u_{i-1})$.
Thus, $M(\ell)=\binom{d}{\ell} a(a-1)^{\ell}$ in \Cref{lem:bip-cross-col}, so
\begin{equation}\label{eq:num-colorings-even-len-path_1}
    N(P; a,b)
    = \sum_{\ell=0}^d T_1(a,b,\ell)
    = \sum_{\ell=0}^d \binom{d}{\ell} a(a-1)^\ell (b-1)^{d-\ell} (b-2)^\ell.
\end{equation}

For a vertex-colouring \(\psi\) counted by \(N(P;b,a)\), also write $L(\psi) \defeq \{2\le i\le d : \psi(v_{i})\neq \psi(v_{i-1})\}$.
This is to $a$-colour $V_2$ so that we can apply~\Cref{lem:bip-cross-col} with bipartition $V_2\sqcup V_1$.
Then the number $T_2(a,b;\ell)$ of those \(\psi\) with \(\abs{L(\psi)}=\ell\) is equal to
\[
    T_2(a,b,\ell) \defeq \binom{d-1}{\ell} a(a-1)^\ell (b-1)^{d-\ell+1} (b-2)^\ell,
\]
since $M(\ell)$ in \Cref{lem:bip-cross-col} is $\binom{d-1}{\ell} a(a-1)^{\ell}$. Indeed, \(\binom{d-1}{\ell}\) counts choices for $L(\psi)$, there are $a$ colour choices for $\psi(v_1)$ and $a-1$ colour choices for $\psi(v_i)$ for $i\in L(\psi)$ avoiding $\psi(u_{i-1})$.
Hence,
\begin{equation}\label{eq:num-colorings-even-len-path_2}
    N(P; b,a)
    = \sum_{\ell=0}^{d-1} T_2(a,b,\ell)
    = \sum_{\ell=0}^{d-1} \binom{d-1}{\ell} a(a-1)^\ell (b-1)^{d-\ell+1} (b-2)^\ell.
\end{equation}

Our goal~\eqref{eq:SBS-bip-goal} is to show
\[
    N(P; a) \, N(P; b) \le N(P; a,b) \, N(P; b,a)
    \qquad\text{for \(2\le a<b\)},
\]
which is equivalent to
\[
   \sum_{\ell=0}^{d} T_1(a,a,\ell)\sum_{\ell=0}^{d-1}T_2(b,b,\ell)
   \leq \sum_{\ell=0}^{d} T_1(a,b,\ell)\sum_{\ell=0}^{d-1}T_2(a,b,\ell)
\]
by \eqref{eq:num-colorings-even-len-path_1} and \eqref{eq:num-colorings-even-len-path_2}. \Cref{prop:sum-prod-red}~\ref{case:(prop:sum-prod-red)-1} then reduces the problem to show that for all \(0\le k\le d\) and \(0\le \ell\le d-1\),
\[
    T_1(a,a,k) \, T_2(b,b,\ell) \le T_1(a,b,k) \, T_2(a,b,\ell).
\]
This is simplified to
\(
    (a-1)^{d-k-\ell} (a-2)^k b \le (b-1)^{d-k-\ell} (b-2)^k a
\),
which follows from $\frac{a-2}{b-2}\leq \frac{a-1}{b-1}\leq \frac{a}{b}$, analogously to \eqref{eq:path-SBS-in-Kq_simp}.

\subsubsection*{Even cycles}
Let \(C\) be a cycle of length \(2d\), \(d\ge2\), with the bipartition \(\{u_i : 1\le i\le d\}\sqcup \{v_i : 1\le i\le d\}\). 
In this graph, we shall take the vertex index sum modulo $d$ and we assume that $u_iv_i$ and $v_iu_{i+1}$ are edges. Let \(2\le a \le b\). For a vertex-colouring \(\phi\) counted by \(N(C;a,b)\), define
\[
    L(\phi) \defeq \{1\le i\le d : \phi(u_{i}) \neq \phi(u_{i-1})\}.
\]
We again count those \(\phi\) with \(\abs{L(\phi)}=\ell\), denoted by $T(a,b,\ell)$, to use~\Cref{lem:bip-cross-col}.
There are \(\binom{d}{\ell}\) choices to choose $L(\phi)$.
Having fixed $L(\phi)$, consider the auxiliary $\ell$-cycle obtained by contracting the edge $u_iu_{i-1}$, \(i\in [d]\setminus L(\phi)\), of the $d$-cycle on $\{u_i:1\le i\le d\}$ with edges $\{u_iu_{i+1}:1\leq i\le d\}$.
Then determining \(\phi(u_i)\), \(1\le i\le d\), is equivalent to properly colouring an \(\ell\)-cycle with \(a\) colours. 
The number of proper $a$-colourings of an $\ell$-cycle is \((a-1)^\ell + (-1)^\ell(a-1)\), which follows from the deletion-contraction formula for chromatic polynomials. Thus, $M(\ell)=\binom{d}{\ell} \big( (a-1)^\ell + (-1)^\ell(a-1) \big)$, so~\Cref{lem:bip-cross-col} gives
\begin{align}\label{eq:T_in_cycles}
    T(a,b,\ell)
    =
    \binom{d}{\ell} \big( (a-1)^\ell + (-1)^\ell(a-1) \big) (b-1)^{d-\ell} (b-2)^\ell
\end{align}
and
\begin{equation}\label{eq:num-colorings-even-cycle}
    N(C;a,b)
    = \sum_{\ell=0}^d T(a,b,\ell)
    = \sum_{\ell=0}^d \binom{d}{\ell} \big( (a-1)^\ell + (-1)^\ell(a-1) \big) (b-1)^{d-\ell} (b-2)^\ell.
\end{equation}

Note that $N(C;a,b)=N(C;b,a)$ due to the part-swapping automorphism of the even cycle $C$. Our goal~\eqref{eq:SBS-bip-goal} is to show
\begin{equation}\label{eq:even-cycle-SBS-in-Kq_goal}
    N(C;a) \, N(C;b)\le N(C;a,b)\, N(C;b,a) = N(C;a,b)^2
    \qquad\text{for \(2\le a < b\)}.
\end{equation}
When \(a=2\), this follows from
\[
    N(C;2) \, N(C;b)
    = 2 \big( (b-1)^d + (-1)^d (b-1) \big)
    \le \big( 2(b-1)^d \big)^2
    \le N(C;2,b)^2.
\]
Now suppose \(a\ge3\). We collect three claims to apply~\Cref{prop:sum-prod-red}.

\begin{claim}\label{clm:even-cycle-SBS-in-Kq_red_case1-2}
    For \(0\le k,\ell\le d\), if $\ell$ is even or \((k,\ell)\neq (0,d)\), then
    \begin{equation}\label{eq:even-cycle-SBS-in-Kq_red_case1}
        T(a,a,k) \, T(b,b,\ell) \le T(a,b,k) \, T(a,b,\ell).
    \end{equation}
\end{claim}
\begin{claimproof}
    By~\eqref{eq:T_in_cycles},
    \begin{equation}\label{eq:even-cycle-SBS-in-Kq_red_case1-2-eq1}
        \frac{\text{LHS of \eqref{eq:even-cycle-SBS-in-Kq_red_case1}}}{\text{RHS of \eqref{eq:even-cycle-SBS-in-Kq_red_case1}}}
        = \biggl(\frac{a-1}{b-1}\biggr)^{d-k} \biggl(\frac{a-2}{b-2}\biggr)^k
            \frac{(b-1)^\ell + (-1)^\ell (b-1)}{(a-1)^\ell + (-1)^\ell (a-1)}.
    \end{equation}
    Suppose first that \(\ell\) is even. Then by \(\frac{a-2}{b-2} \le \frac{a-1}{b-1}\), it suffices to show that
    \begin{equation}\label{eq:even-cycle-SBS-in-Kq_red_case1-2-eq2}
        \biggl(\frac{a-1}{b-1}\biggr)^{d}
            \frac{(b-1)^\ell + (b-1)}{(a-1)^\ell + (a-1)}\leq 1.
    \end{equation}
    If \(\ell=0\), then
    \(
        \eqref{eq:even-cycle-SBS-in-Kq_red_case1-2-eq2}
        \le \bigl(\frac{a-1}{b-1}\bigr)^{d} \frac{b}{a}
        \le \frac{a-1}{b-1} \frac{b}{a}
        \le 1
    \); if \(\ell>0\) is even, then
    \(
        \eqref{eq:even-cycle-SBS-in-Kq_red_case1-2-eq2}
        \le \bigl(\frac{a-1}{b-1}\bigr)^{d} \frac{(b-1)^\ell}{(a-1)^\ell}
        \le 1
    \).
    Second, suppose \(\ell\) is odd. If \(\ell<d\), then
    \[
        \eqref{eq:even-cycle-SBS-in-Kq_red_case1-2-eq1}
        \le \biggl(\frac{a-1}{b-1}\biggr)^{d}
                \frac{(b-1)^\ell - (b-1)}{(a-1)^\ell - (a-1)}
        \le \biggl(\frac{a-1}{b-1}\biggr)^{d} \frac{(b-1)^{\ell+1}}{(a-1)^{\ell+1}}
        \le 1,
    \]
    where the first inequality uses \(\frac{a-2}{b-2} \le \frac{a-1}{b-1}\) and the second inequality follows from the fact that \((x^\ell-x)/x^{\ell+1}\) is decreasing for \(x\ge 2\).
    Otherwise, if \(\ell=d\), then \(k>0\) by the assumption. In this case,
    \[
        \eqref{eq:even-cycle-SBS-in-Kq_red_case1-2-eq1}
        \le \biggl(\frac{a-1}{b-1}\biggr)^{d-1} \biggl(\frac{a-2}{b-2}\biggr)
                \frac{(b-1)^d - (b-1)}{(a-1)^d - (a-1)}
        \le 1,
    \]
    where the first inequality uses \(\frac{a-2}{b-2} \le \frac{a-1}{b-1}\) and the second inequality follows from the fact that \((x^d-x)/(x^{d-1}(x-1))\) is decreasing for \(x\ge2\).
\end{claimproof}

\begin{claim}\label{clm:even-cycle-SBS-in-Kq_red_case1}
    For \(0\le k,\ell\le d\), if \(d\) is even, then
    \[
        T(a,a,k) \, T(b,b,\ell) \le T(a,b,k) \, T(a,b,\ell).
    \]
\end{claim}
\begin{claimproof}
    If \(\ell\) is even, this follows from \Cref{clm:even-cycle-SBS-in-Kq_red_case1-2}; if \(\ell\) is odd, then \(\ell<d\) since \(d\) is even, so again the conclusion follows from \Cref{clm:even-cycle-SBS-in-Kq_red_case1-2}.
\end{claimproof}
It then remains to analyse the `boundary case' $(k,\ell)=(0,d)$ for odd $d$. This particular case must be handled by pairing it with $(k,\ell)=(d,0)$.

\begin{claim}\label{clm:even-cycle-SBS-in-Kq_red_case2}
    For odd \(d\),
    \begin{equation}\label{eq:even-cycle-SBS-in-Kq_red_case2}
        T(a,a,0) \, T(b,b,d) + T(a,a,d) \, T(b,b,0)  \le T(a,b,0) \, T(a,b,d) + T(a,b,d) \, T(a,b,0).
    \end{equation}
\end{claim}
\begin{claimproof}
    Since \(d\geq 2\) is odd, \(d\ge3\).
    We shall use the elementary bound
    \begin{equation}\label{eq:C_2d-SBS-red2-step2}
        a\le (a+1) \biggl( 1-\frac{1}{(a-1)^2} \biggr) \le (a+1) \biggl( 1-\frac{1}{(a-1)^{d-1}} \biggr),
    \end{equation}
    which follows from $a(a-1)^2 - (a+1)\big((a-1)^2-1\big) = a(3-a)\le 0$.
    We also use
    \begin{equation}\label{eq:C_2d-SBS-red2-step1}
        a - \frac{(a-2)^d}{(b-2)^d} b
        \ge a - \frac{(a-2)^2}{(b-2)^2} b
        \ge 1,
    \end{equation}
    which follows from
    $a(b-2)^2 - (a-2)^2 b  - (b-2)^2
        = ab(b-a-1) + 3b(b-a) + 4(a-1) \ge 0$.
    Therefore,
    \[
        a \biggl( 1-\frac{1}{(b-1)^{d-1}} \biggr)
        \le a
        \overset{\eqref{eq:C_2d-SBS-red2-step2}}{\le} (a+1) \biggl( 1-\frac{1}{(a-1)^{d-1}} \biggr)
        \overset{\eqref{eq:C_2d-SBS-red2-step1}}{\le} \biggl( 2a - \frac{(a-2)^d}{(b-2)^d} b \biggr) \biggl( 1-\frac{1}{(a-1)^{d-1}} \biggr).
    \]
    Thus,
    \[
        a(a-1)^d\big( (b-1)^d - (b-1) \big) (b-2)^d
        \le \big( (a-1)^d - (a-1) \big)\big( 2a(b-2)^d - (a-2)^d b \big) (b-1)^d,
    \]
    which proves \eqref{eq:even-cycle-SBS-in-Kq_red_case2}.
\end{claimproof}

We now substitute \eqref{eq:num-colorings-even-cycle} into every term of \eqref{eq:even-cycle-SBS-in-Kq_goal}. If \(d\) is even, then \Cref{prop:sum-prod-red}~\ref{case:(prop:sum-prod-red)-1} and \Cref{clm:even-cycle-SBS-in-Kq_red_case1} verify the inequality. Otherwise, if \(d\) is odd, we slightly modify the sum to apply \Cref{prop:sum-prod-red}~\ref{case:(prop:sum-prod-red)-1} by pairing two terms in the boundary into one, i.e.,
\begin{gather}
    T(a,a,k) \, T(b,b,\ell) \le T(a,b,k) \, T(a,b,\ell)
    \qquad\text{for } 0\le k,\ell\le d,\ (k,\ell)\neq (0,d), (d,0);
    \label{eq:even-cycle-SBS-in-Kq_red_d-even_case1} \\
    T(a,a,0) \, T(b,b,d) + T(a,a,d) \, T(b,b,0)  \le T(a,b,0) \, T(a,b,d) + T(a,b,d) \, T(a,b,0),
    \label{eq:even-cycle-SBS-in-Kq_red_d-even_case2}
\end{gather}
which follows from \Cref{clm:even-cycle-SBS-in-Kq_red_case1-2,clm:even-cycle-SBS-in-Kq_red_case2}, respectively. This proves \eqref{eq:even-cycle-SBS-in-Kq_goal}.

\subsection{Complete multipartite graphs}\label{sec:complete_multipartite}
In this subsection, we show that the complete \(k\)-partite graphs are cross-bipartite swapping in the complete graphs.
For nonnegative integers \(r_1,\dots,r_k\),
let \(K(r_1,\dots,r_k)\) be the multipartite graph with the part sizes \(r_1,\dots,r_k\).
In particular, $r_1=r_2=\cdots=r_k=1$ makes the complete graph $K_k$ and $K(r)$ denotes the $r$ isolated vertices. 

Let \(A, B\subseteq K_q\). By~\Cref{lem:SBS-in-Kq_simp-condition}, we may assume that \(A\subseteq B\) or \(B\subseteq A\), so the homomorphism counts depend only depend on $a\defeq|A|$ and $b\defeq|B|$ with $a\leq b$.
For brevity, write
\[
    N_k(r_1,\dots,r_k; a) \defeq N(K(r_1,\dots,r_k); a) = \hom(K(r_1,\dots,r_k), K_q[A]).
\]
Our goal~\eqref{eq:SBS-goal} is to show that for all positive integers \(a\le b\) and \(r_1,\dots,r_k\),
\begin{equation}\label{eq:SBS-goal-comp-mpart}
    N_k(r_1,\dots,r_k; a) \, N_k(r_1,\dots,r_k; b)
    \le N(K(r_1,\dots,r_k)\times K_2; a,b).
\end{equation}
To show \eqref{eq:SBS-goal-comp-mpart}, we use induction on \(k\).
For the induction step, the following lemma that resembles \Cref{cor:G-vol} will be useful. Denote by \(N_k^{(1)}(r_1,\dots,r_k; a,b)\) the number of homomorphic copies of \(K(r_1,\dots,r_k)\) in \(K_q\), where the first part of size \(r_1\) maps into \(A\) and the other parts map into \(B\).

\begin{lemma}\label{lem:comp-mpart_weak-AF}
    For positive integers \(a\le b\), \(k\ge2\), and \(r_1,\dots,r_k\),
    \[
        N_k(r_1,\dots,r_k; a) \, N_k(r_1,\dots,r_k; b)
        \le N_k^{(1)}(r_1,\dots,r_k; a,b) \, N_k^{(1)}(r_1,\dots,r_k; b,a).
    \]
\end{lemma}

Setting \(r_1=\dots=r_k=1\) in \Cref{lem:comp-mpart_weak-AF}, we obtain
\[
    V_{K_k}(\one_A,\dots,\one_A; K_q)
    \cdot V_{K_k}(\one_B,\dots,\one_B; K_q)
    \le
    V_{K_k}(\one_A, \underbrace{\one_B,\dots,\one_B}_{k-1}; K_q)
    \cdot V_{K_k}(\one_B, \underbrace{\one_A,\dots,\one_A}_{k-1}; K_q),
\]
which is exactly \Cref{cor:G-vol} with \(t=k\), \(G=K_q\), \(\ba=\one_A\), and \(\bb=\one_B\).
In fact, we use \Cref{lem:comp-mpart_weak-AF} in the induction step in the same way that \Cref{cor:G-vol} is used in the proof of \Cref{thm:G-vol-ineq-KtxK2}.

\begin{proof}[Proof of~\eqref{eq:SBS-goal-comp-mpart} using \Cref{lem:comp-mpart_weak-AF}]
As mentioned, we use induction on \(k\). The base case $k=1$ is clear as
\[
    N_1(r_1; a) \, N_1(r_1; b)
    = a^{r_1} b^{r_1}
    = N(K(r_1)\times K_2; a,b).
\]
Suppose \(k\ge2\) and \eqref{eq:SBS-goal-comp-mpart} holds for all \(a\le b\), \(r_1,\dots,r_k\), and \(k\) replaced by \(k-1\).

Let \(V_1\) and $V_1'$ be the vertex set of \(K(r_1,\dots,r_k)\times K_2\) of size \(r_1\), which correspond to two copies of the first part of $K(r_1,\dots,r_k)$ of size $r_1$.
Each homomorphism counted by \(N(K(r_1,\dots,r_k)\times K_2; a,b)\) can be obtained by fixing two maps \(\psi\colon V_1\to A\) and \(\psi'\colon V'_1\to B\) and finding a homomorphic copy of \(K(r_2,\dots,r_k)\times K_2\) in \(K_q\) by mapping the two parts of its canonical bipartition into \(A_{\psi'}\defeq A\setminus\psi'(V'_1)\) and \(B_{\psi}\defeq B\setminus\psi(V_1)\), respectively. Therefore,
\begin{align}
    N(K(r_1,\dots,r_k)\times K_2; a,b)
        \nonumber
    &= \sum_{\substack{\psi\colon V_1\to A \\ \psi'\colon V'_1\to B}}
        \bhom \bigl( K(r_2,\dots,r_k)\times K_2, \, K_q[A_{\psi'},B_{\psi}] \bigr)
        \nonumber
    \\
    &\ge \sum_{\substack{\psi \colon V_1\to A \\ \psi' \colon V'_1\to B}}
        N\bigl( K(r_2,\dots,r_k)\times K_2; \, \abs{ A_{\psi'} }, \abs{ B_{\psi} } \bigr).
        \label{eq:SBS-goal-comp-mpart_step1}
\end{align}
To see why the inequality holds, recall that in~\eqref{eq:N(H;a,b)_reduction}, we have shown that $\bhom(H\times K_2, K_q[A,B])$ decreases as we assume $A\subseteq B$, which corresponds to the definition of $N(H\times K_2;a,b)$.
By the induction hypothesis, \eqref{eq:SBS-goal-comp-mpart_step1} is at least
\begin{align}
    &\sum_{\substack{\psi\colon V_1\to A \\ \psi'\colon V'_1\to B}}
        N_{k-1}\bigl( r_2,\dots,r_k ; \abs{ B_{\psi} }\bigr)
            \cdot N_{k-1}\bigl( r_2,\dots,r_k ; \abs{ A_{\psi'} }\bigr)
        \nonumber
    \\ ={} & \sum_{\psi\colon V_1\to A} N_{k-1}\bigl( r_2,\dots,r_k ; \abs{ B_{\psi} } \bigr)
        \cdot \sum_{\psi'\colon V'_1\to B} N_{k-1}\bigl( r_2,\dots,r_k ; \abs{ A_{\psi'} } \bigr)
        \nonumber
    \\ ={} & \sum_{\psi\colon V_1\to A} \hom \bigl( K(r_2,\dots,r_k), K_q[B_{\psi}] \bigr)
        \cdot \sum_{\psi'\colon V'_1\to B} \hom \bigl( K(r_2,\dots,r_k), K_q[A_{\psi'}] \bigr).
        \label{eq:SBS-goal-comp-mpart_step2}
\end{align}
Since the first sum in \eqref{eq:SBS-goal-comp-mpart_step2} counts the number of homomorphisms from \(K(r_1,\dots,r_k)\) to \(K_q\) which maps \(V_1\) into \(A\) via~$\psi$ and the other parts into \(B\) while avoiding colours used by $\psi$, the sum $\sum_{\psi\colon V_1\to A} \hom \bigl( K(r_2,\dots,r_k), K_q[B_{\psi}] \bigr)$ equals $N_k^{(1)}(r_1,\dots,r_k; a,b)$. By the same reason, the second sum in \eqref{eq:SBS-goal-comp-mpart_step2} equals \(N_k^{(1)}(r_1,\dots,r_k; b,a)\). Finally, we apply \Cref{lem:comp-mpart_weak-AF} to obtain the desired inequality
\[
    N(K(r_1,\dots,r_k)\times K_2; a,b)
    \ge
    N_k(r_1,\dots,r_k; a) \, N_k(r_1,\dots,r_k; b),
\]
which complete the induction.
\end{proof}

It remains to prove \Cref{lem:comp-mpart_weak-AF}. Unlike~\Cref{cor:G-vol}, $H=K(r_1,\dots,r_k)$ does not give the Lorentzian property of $h_H(\bx;K_q)$ in general; for more discussions, see the concluding remarks. This forces us to take more hands-on approaches.

\begin{proof}[Proof of \Cref{lem:comp-mpart_weak-AF}]
    Fix \(k\ge 2\). We shall prove the following slightly stronger statement:
    \begin{equation}\label{eq:comp-mpart_weak-AF_stronger}
        \begin{aligned}
            &%
            \text{For all \(k\le a\le b\) and nonnegative integers \(r_i\le s_i\), \(1\le i\le k\),}
            \\& N_k(s_1,\dots,s_k; a) \, N_k(r_1,\dots,r_k; b)
                \le N_k^{(1)}(s_1,r_2,\dots,r_k; a,b) \, N_k^{(1)}(r_1,s_2,\dots,s_k; b,a).
        \end{aligned}
    \end{equation}
   Here we assume \(k\le a\le b\) since \(a<k\) makes the LHS vanish. Allowing $s_i$ or $r_i$ to be zero is necessary for the induction, in particular, in~\Cref{clm:(comp-mpart_weak-AF)-pf_reduction1}; for instance, $N_k(0,r_2,\dots,r_k;b)$ means $N_{k-1}(r_2,\dots,r_k;b)$.

   The proof idea relies on multiple reduction steps, on which we briefly give a sketch first. 
   First, using induction on \(b\), where the induction step is given in \Cref{clm:(comp-mpart_weak-AF)-pf_reduction1}, we reduce the problem to the base case when \(b=a\), i.e.,
   \begin{equation}\label{eq:comp-mpart_weak-AF_stronger_b=a}
        \begin{aligned}
            &%
            \text{for all \(a \ge k\) and \(0\le r_i\le s_i\), \(1\le i\le k\),}
            \\& N_k(s_1,\dots,s_k; a) \, N_k(r_1,\dots,r_k; a)
                \le N_k(s_1,r_2,\dots,r_k; a) \, N_k(r_1,s_2,\dots,s_k; a).
        \end{aligned}
    \end{equation}
    Second, a telescoping argument in \Cref{clm:(comp-mpart_weak-AF)-pf_reduction(-4)} reduces this to the case when \(s_1=r_1+1\), \(s_2=r_2+1\), and \(s_i=r_i\) for \(3\le i\le k\), i.e.,
    \begin{align*}
        &%
        \text{for all \(a\ge k\) and \(r_i\ge 0\), \(1\le i\le k\),}
        \\&\hphantom{{}\le{}} N_k(r_1+1,r_2+1, r_3\dots,r_k; a) \, N_k(r_1,r_2, r_3\dots,r_k; a)
        \\&\le N_k(r_1+1,r_2, r_3\dots,r_k; a) \, N_k(r_1,r_2+1, r_3\dots,r_k; a).
    \end{align*}
    Third, rewriting each term \(N_k(\blank)\) as a sum of terms on \(N_{k-1}(\blank)\) allows us to apply \Cref{prop:sum-prod-red}~\ref{case:(prop:sum-prod-red)-2}, which reduces the problem further to the following statement, shown in \Cref{clm:(comp-mpart_weak-AF)-pf_reduction(-3)}:
    \begin{align*}
        &\text{for all \(a\ge k\), \(0\le \ell\le m\le a\), and \(r_i\ge 0\), \(2\le i\le k\),}
        \\&\hphantom{{}\le{}}
            N_{k-1}(r_2, r_3,\dots,r_k; a-\ell) \, N_{k-1}(r_2+1, r_3,\dots,r_k; a-m)
        \\&\le N_{k-1}(r_2+1, r_3,\dots,r_k; a-\ell) \, N_{k-1}(r_2, r_3,\dots,r_k; a-m).
    \end{align*}
    Fourth, another telescoping argument in \Cref{clm:(comp-mpart_weak-AF)-pf_reduction(-2)} will reduce this to the case $m=\ell+1$, i.e., 
    \begin{align*}
        &\text{for all \(a\ge k-1\) and \(r_i\ge 0\), \(2\le i\le k\),}
        \\&\hphantom{{}\le{}}
            N_{k-1}(r_2, r_3,\dots,r_k; a+1) \, N_{k-1}(r_2+1,r_3,\dots,r_k; a)
        \\&\le N_{k-1}(r_2+1,r_3,\dots,r_k; a+1) \, N_{k-1}(r_2, r_3,\dots,r_k; a).
    \end{align*}
    Finally, \Cref{clm:(comp-mpart_weak-AF)-pf_reduction(-1)} will prove this by writing each term of the form \(N_{k-1}(\blank)\) as a sum of terms on \(N_{k-2}(\blank)\) and applying \Cref{prop:sum-prod-red}~\ref{case:(prop:sum-prod-red)-2}.

    \medskip
    
    We begin with the first reduction step.
    
    \begin{claim}\label{clm:(comp-mpart_weak-AF)-pf_reduction1}
        \eqref{eq:comp-mpart_weak-AF_stronger_b=a} implies \eqref{eq:comp-mpart_weak-AF_stronger}.
    \end{claim}
    \begin{claimproof}
        It suffices to show that \eqref{eq:comp-mpart_weak-AF_stronger} for fixed $(k,a,b)$ with \(k\le a\le b\) implies \eqref{eq:comp-mpart_weak-AF_stronger} for $(k,a,b+1)$.
        Note first that
        \begin{equation}\label{eq:comp-mpart_weak-AF_red1_b+1_eq1}
            N_k(r_1,\dots,r_k; b+1)
            = N_k(r_1,\dots,r_k; b)
                + \sum_{i=1}^k \sum_{\ell=1}^{r_i} \binom{r_i}{\ell} N_k(r_1,\dots, r_i-\ell, \dots,r_k; b),
        \end{equation}
        where \(N_k(r_1,\dots,r_i-\ell,\dots,r_k;b)\) denotes \(N_k(r_1,\dots, r_k; b)\) with \(r_i\) replaced by \(r_i-\ell\).
        Indeed, the term \(N_k(r_1,\dots,r_k; b+1)\) counts the number of proper vertex-colourings \(K(r_1,\dots,r_k)\) using \(b+1\) colours. Fix one of the $b+1$ colours, denoted by \(\nu\). Let \(V_1\sqcup\dots\sqcup V_k\) be the $k$-partition of \(K(r_1,\dots,r_k)\) where \(\abs{V_i}=r_i\) for each \(1\le i\le k\). 
        The number of colourings that do not use \(\nu\) is \(N_k(r_1,\dots,r_k; b)\). Otherwise, $\nu$ must be used in exactly one part and the number of colourings that use \(\nu\) exactly \(\ell\) times in \(V_i\) is \(\binom{r_i}{\ell} N_k(r_1,\dots, r_i-\ell, \dots,r_k; b)\), for each \(1\le i\le k\) and \(1\le\ell\le r_i\). 

        An analogous argument also gives
        \begin{equation}\label{eq:comp-mpart_weak-AF_red1_b+1_eq2}
            N_k^{(1)}(s_1,r_2,\dots,r_k; a,b+1)
            = N_k^{(1)}(s_1,r_2,\dots,r_k; a,b) + \Sigma_1,
        \end{equation}
        where
        \[
            \Sigma_1 \defeq
            \sum_{i=2}^k \sum_{\ell=1}^{r_i} \binom{r_i}{\ell} N_k^{(1)}(s_1,r_2,\dots,r_i-\ell,\dots,r_k; a,b).
        \]
        Here, \(N_k^{(1)}(s_1,r_2,\dots,r_i-\ell,\dots,r_k; a,b)\) denotes \(N_k^{(1)}(s_1,r_2,\dots,r_k; a,b)\) with \(r_i\) replaced by \(r_i-\ell\). 
        The key difference between \eqref{eq:comp-mpart_weak-AF_red1_b+1_eq1} and \eqref{eq:comp-mpart_weak-AF_red1_b+1_eq2} is that we choose the colour $\nu$ not amongst the $a$ colours and the sum analyses how all the parts but the $s_1$-vertex part use the colour $\nu$.
        Similarly,
        \begin{equation}\label{eq:comp-mpart_weak-AF_red1_b+1_eq3}
            N_k^{(1)}(r_1,s_2,\dots,s_k; b+1,a)
            = N_k^{(1)}(r_1,s_2,\dots,s_k; b,a) + \Sigma_2,
        \end{equation}
        where
        \[
            \Sigma_2 \defeq \sum_{\ell=1}^{r_1} \binom{r_1}{\ell} N_k^{(1)}(r_1-\ell, s_2,\dots,s_k; b,a).
        \]
        Indeed, this is obtained by analysing how $\nu$ can be used in the first $r_1$-vertex part.        

        Now we show \eqref{eq:comp-mpart_weak-AF_stronger} for $(k,a,b+1)$. Applying \eqref{eq:comp-mpart_weak-AF_red1_b+1_eq1} and expanding out yields
        \begin{align}
            & N_k(s_1,\dots,s_k; a) \, N_k(r_1,\dots,r_k; b+1) \nonumber
            \\={} & N_k(s_1,\dots,s_k; a) \, N_k(r_1,\dots,r_k; b)
                    + \sum_{i=1}^k \sum_{\ell=1}^{r_i}
                        \binom{r_i}{\ell}
                        N_k(s_1,\dots,s_k; a) \, N_k(r_1,\dots, r_i-\ell, \dots,r_k; b).
                    \label{eq:comp-mpart_weak-AF_red1_ind-applicant}
        \end{align}
        Since \eqref{eq:comp-mpart_weak-AF_stronger} is assumed to hold for \((k,a,b)\), \eqref{eq:comp-mpart_weak-AF_red1_ind-applicant} is at most
        \begin{align*}
            & N_k^{(1)}(s_1,r_2,\dots,r_k; a,b) \, N_k^{(1)}(r_1,s_2,\dots,s_k; b,a)
            \\ {}+{} & \sum_{\ell=1}^{r_1} \binom{r_1}{\ell}
                N_k^{(1)}(s_1,r_2,\dots,r_k; a,b) \, N_k^{(1)}(r_1-\ell,s_2,\dots,s_k; b,a)          
            \\ {}+{} & \sum_{i=2}^k \sum_{\ell=1}^{r_i} \binom{r_i}{\ell}
                N_k^{(1)}(s_1,r_2,\dots,r_i-\ell,\dots,r_k; a,b) \, N_k^{(1)}(r_1,s_2,\dots,s_k; b,a),
        \end{align*}
        where we separate the case $i=1$ in~\eqref{eq:comp-mpart_weak-AF_red1_ind-applicant} from the others to give the upper bound.
        This rewrites as
        \begin{align*}
            & N_k^{(1)}(s_1,r_2,\dots,r_k; a,b) \, N_k^{(1)}(r_1,s_2,\dots,s_k; b,a)
            \\ {}+{} & N_k^{(1)}(s_1,r_2,\dots,r_k; a,b) \, \Sigma_2
            %\\ {}+{} &
            + N_k^{(1)}(r_1,s_2,\dots,s_k; b,a) \, \Sigma_1,
        \end{align*}
        which is the product of \eqref{eq:comp-mpart_weak-AF_red1_b+1_eq2} and \eqref{eq:comp-mpart_weak-AF_red1_b+1_eq3} minus $\Sigma_1\Sigma_2$. Therefore, the above is at most
        \[
            N_k^{(1)}(s_1,r_2,\dots,r_k; a,b+1) \, N_k^{(1)}(r_1,s_2,\dots,s_k; b+1,a),
        \]
        which concludes the proof.
    \end{claimproof}

    Before getting into the remaining steps, we see how \(N_k(\blank)\) can be written as a sum of terms on \(N_{k-1}(\blank)\). This will be used in the proofs of \Cref{clm:(comp-mpart_weak-AF)-pf_reduction(-1),clm:(comp-mpart_weak-AF)-pf_reduction(-3)}.

    \begin{claim}\label{clm:comp-mpart_Nk-into-N{k-1}}
        For nonnegative integers $r_1,r_2,\dots,r_k$,
        \[
            N_k(r_1,r_2,\dots,r_k; a)
            = \sum_{\ell=0}^a \binom{a}{\ell} \stirling{r_1}{\ell} \ell!\, N_{k-1}(r_2,\dots,r_k; a-\ell).\footnote{\(\stirling{r}{\ell}\) is the Stirling number of the second kind---the number of ways to partition \([r]\) into \(\ell\) nonempty subsets.}
        \]
    \end{claim}
    \begin{claimproof}
        Let \(V_1\sqcup \cdots \sqcup V_k\) be the $k$-partition of \(K(r_1,\dots,r_k)\) where \(\abs{V_i}=r_i\) for \(1\le i\le k\).
        For \(0\le \ell\le a\), the number of vertex-colourings \(\phi\) counted by \(N_k(r_1,r_2,\dots,r_k; a)\) with \(\abs{\phi(V_1)}=\ell\) is
        \[
            T(r_1,r_2,\ell)
            \defeq 
            \binom{a}{\ell} \stirling{r_1}{\ell} \ell!\, N_{k-1}(r_2,\dots,r_k; a-\ell).
        \]
        Indeed, there are \(\binom{a}{\ell}\) choices to choose the set \(\phi(V_1)\), and \(\stirling{r}{\ell} \ell!\) ways to colour \(V_1\) using all the colours from \(\phi(V_1)\).
        Having fixed the values of \(\phi\) on \(V_1\), there are \(N_{k-1}(r_2,\dots,r_k; a-\ell)\) ways to colour \(K(r_1,r_2,\dots,r_k)\setminus V_1\cong K(r_2,r_3,\dots,r_k)\) using colours not in \(\phi(V_1)\).
        Thus,
        \begin{equation}\label{eq:comp-mpart_weak-AF_count1}
            N_k(r_1,r_2,\dots,r_k; a)
            = \sum_{\ell=0}^a T(r_1,r_2,\ell)
            = \sum_{\ell=0}^a \binom{a}{\ell} \stirling{r_1}{\ell} \ell!\, N_{k-1}(r_2,\dots,r_k; a-\ell),
        \end{equation}
        as claimed.
    \end{claimproof}

    The remaining parts of the proof build upon `log-submodularity' results on $N_{k}(\blank)$, which formalise the aforementioned reduction steps in the reverse order. In a sense, our goal~\eqref{eq:comp-mpart_weak-AF_stronger_b=a} itself can also be seen as a log-submodularity result.

    \begin{claim}\label{clm:(comp-mpart_weak-AF)-pf_reduction(-1)}
        For all \(a\ge k-1\) and \(r_2,\dots,r_k\ge 0\),
        \begin{equation}\label{eq:(comp-mpart_weak-AF)-pf_reduction(-1)}
        \begin{aligned}
            & N_{k-1}(r_2, r_3,\dots,r_k; a+1) \, N_{k-1}(r_2+1,r_3,\dots,r_k; a)
            \\ \le{} & N_{k-1}(r_2+1,r_3,\dots,r_k; a+1) \, N_{k-1}(r_2, r_3,\dots,r_k; a).
        \end{aligned}
        \end{equation}
    \end{claim}
    \begin{claimproof}  
        Throughout the proof, $r_3,\dots,r_k$ are fixed.
        Let $T_1(r, a, \ell) \defeq \binom{a}{\ell} \stirling{r}{\ell} \ell!\, N_{k-2}(r_3,\dots,r_k; a-\ell)$ so that \Cref{clm:comp-mpart_Nk-into-N{k-1}} with \(k\) replaced by \(k-1\) gives
        \[
            N_{k-1}(r,r_3,\dots,r_k; a) = \sum_{\ell=0}^a T_1(r, a, \ell),
        \]
        with the convention that \(N_0(r_3,\dots,r_k;a-\ell)\defeq 1\).
        Substituting this into every term in \eqref{eq:(comp-mpart_weak-AF)-pf_reduction(-1)} with $r=r_2$ or $r_2+1$, it suffices to show that, by \Cref{prop:sum-prod-red}~\ref{case:(prop:sum-prod-red)-2}, for all \(0\le \ell\le m\le a\),
        \begin{align*}
            &\hphantom{{}\le{}}%
            T_1(r_2, a+1, \ell) \, T_1(r_2+1, a, m)
                + T_1(r_2, a+1, m) \, T_1(r_2+1, a, \ell)
            \\&\le T_1(r_2+1, a+1, \ell) \, T_1(r_2, a, m)
                + T_1(r_2+1, a+1, m) \, T_1(r_2, a, \ell).
        \end{align*}
        This reduces to
        \[
            \biggl[ \binom{a+1}{\ell} \binom{a}{m} - \binom{a+1}{m} \binom{a}{\ell} \biggr]
            \times \biggl[ \stirling{r_2+1}{\ell} \stirling{r_2}{m} - \stirling{r_2+1}{m} \stirling{r_2}{\ell} \biggr]
            \ge 0.
        \]
        Simple calculation shows that the first term involving binomial coefficients is nonpositive, and the second term is nonpositive by log-concavity of Stirling numbers, e.g.,~\cite[Theorem~3.2]{sibuya1988logconcavity}, so the nonnegativity of the product of the two follows.
    \end{claimproof}

    \begin{claim}\label{clm:(comp-mpart_weak-AF)-pf_reduction(-2)}
        For all \(a\ge k\), \(0\le \ell\le m\le a\) and \(r_2,\dots,r_k\ge 0\),
        \begin{equation}\label{eq:(comp-mpart_weak-AF)-pf_reduction(-2)}
        \begin{aligned}
            & N_{k-1}(r_2, r_3,\dots,r_k; a-\ell) \, N_{k-1}(r_2+1, r_3,\dots,r_k; a-m)
            \\ \le{} & N_{k-1}(r_2+1, r_3,\dots,r_k; a-\ell) \, N_{k-1}(r_2, r_3,\dots,r_k; a-m).
        \end{aligned}
        \end{equation}
    \end{claim}
    \begin{claimproof}
        Rearranging \eqref{eq:(comp-mpart_weak-AF)-pf_reduction(-1)} gives 
        \[
            \frac {N_{k-1}(r_2+1, r_3,\dots,r_k; a)}
                {N_{k-1}(r_2, r_3,\dots,r_k; a)}
            \le \frac {N_{k-1}(r_2+1,r_3,\dots,r_k; a+1)}
                {N_{k-1}(r_2, r_3,\dots,r_k; a+1)},
        \]
        where the RHS is obtained from the LHS by increasing the number of colours \(a\) by \(1\).
        Repeatedly applying this yields that for all \(k-1\le a\le b\),
        \[
            \frac {N_{k-1}(r_2+1, r_3,\dots,r_k; a)}
                {N_{k-1}(r_2, r_3,\dots,r_k; a)}
            \le \frac {N_{k-1}(r_2+1,r_3,\dots,r_k; b)}
                {N_{k-1}(r_2, r_3,\dots,r_k; b)}.
        \]
        This confirms \eqref{eq:(comp-mpart_weak-AF)-pf_reduction(-2)} whenever \(a-m\ge k-1\). If \(a-m< k-1\), then both \(N_{k-1}(r_2+1, r_3,\dots,r_k; a-m)\) and \(N_{k-1}(r_2, r_3,\dots,r_k; a-m)\) are zero, in which case \eqref{eq:(comp-mpart_weak-AF)-pf_reduction(-2)} also holds.
    \end{claimproof}

    \begin{claim}\label{clm:(comp-mpart_weak-AF)-pf_reduction(-3)}
        For all \(a\ge k\) and \(r_1,\dots,r_k\ge 0\),
        \begin{equation}\label{eq:(comp-mpart_weak-AF)-pf_reduction(-3)}
        \begin{aligned}
            & N_k(r_1+1,r_2+1, r_3\dots,r_k; a) \, N_k(r_1,r_2, r_3\dots,r_k; a)
            \\ \le{} & N_k(r_1+1,r_2, r_3\dots,r_k; a) \, N_k(r_1,r_2+1, r_3\dots,r_k; a).
        \end{aligned}
        \end{equation}
    \end{claim}
    \begin{claimproof}
        Throughout the proof, \(r_3,\dots,r_k\) and \(a\) are fixed. Let 
        \[
            T_2(r,s,\ell) \defeq \binom{a}{\ell} \stirling{r}{\ell} \ell!\, N_{k-1}(s, r_3,\dots,r_k; a-\ell)
        \]
        so that \Cref{clm:comp-mpart_Nk-into-N{k-1}} gives
        \[
            N_k(r,s,r_3,\dots,r_k; a)
            = \sum_{\ell=0}^a T_2(r,s,\ell).
        \]
        Substituting this to every term of \eqref{eq:(comp-mpart_weak-AF)-pf_reduction(-3)}, \Cref{prop:sum-prod-red}~\ref{case:(prop:sum-prod-red)-2} reduces~\eqref{eq:(comp-mpart_weak-AF)-pf_reduction(-3)} to proving that, for all \(0\le \ell\le m\le a\),
        \begin{align*}
            &\hphantom{{}\le{}}
            T_2(r_1+1,r_2+1,\ell) \, T_2(r_1,r_2, m) + T_2(r_1+1,r_2+1,m) \, T_2(r_1, r_2, \ell)
            \\&\le T_2(r_1+1,r_2, \ell) \, T_2(r_1, r_2+1,m) + T_2(r_1+1,r_2, m) \, T_2(r_1,r_2+1,\ell).
        \end{align*}
        This simplifies to
        \begin{align*}
            &\biggl[ \stirling{r_1+1}{\ell} \stirling{r_1}{m} - \stirling{r_1+1}{m} \stirling{r_1}{\ell} \biggr]
            \\ {}\times{} & \Bigl[ N_{k-1}(r_2+1,r_3,\dots,r_k; a-\ell) \, N_{k-1}(r_2,r_3,\dots,r_k; a-m)
            \\ &\hspace{2em} {}- N_{k-1}(r_2,r_3,\dots,r_k; a-\ell) \, N_{k-1}(r_2+1,r_3,\dots,r_k; a-m) \Bigr]
            \le 0.
        \end{align*}
        The term \(\stirling{r_1+1}{\ell} \stirling{r_1}{m} - \stirling{r_1+1}{m} \stirling{r_1}{\ell}\) is nonpositive, again by~\cite[Theorem~3.2]{sibuya1988logconcavity}, and the other term is nonnegative by \Cref{clm:(comp-mpart_weak-AF)-pf_reduction(-2)}. This concludes the proof of the claim.
    \end{claimproof}
    
    Our last claim proves \eqref{eq:comp-mpart_weak-AF_stronger_b=a}, which we restate for the reader's convenience.
    
    \begin{claim}\label{clm:(comp-mpart_weak-AF)-pf_reduction(-4)}
        For all \(a\ge k\) and \(0\le r_i\le s_i\), \(1\le i\le k\),
        \[
            N_k(s_1,\dots,s_k; a) \, N_k(r_1,\dots,r_k; a)
            \le N_k(s_1,r_2,\dots,r_k; a) \, N_k(r_1,s_2,\dots,s_k; a).
        \]
    \end{claim}
    \begin{claimproof}
        Rearrange \eqref{eq:(comp-mpart_weak-AF)-pf_reduction(-3)} as
        \[
            \frac{N_k(r_1, r_2, r_3, \dots,r_k; a)}
                {N_k(r_1, r_2+1, r_3, \dots,r_k; a)}
            \le \frac{N_k(r_1+1, r_2, r_3, \dots,r_k; a)}
                {N_k(r_1+1, r_2+1, r_3, \dots,r_k; a)}
        \]
        so that the RHS is obtained from the LHS by increasing the first parameter \(r_1\) by \(1\) while fixing the others. Repeatedly applying this yields that for all \(s_1\ge r_1\),
        \[
            \frac{N_k(r_1, r_2, r_3, \dots,r_k; a)}
                {N_k(r_1, r_2+1, r_3, \dots,r_k; a)}
            \le \frac{N_k(s_1, r_2, r_3, \dots,r_k; a)}
                {N_k(s_1, r_2+1, r_3 ,\dots,r_k; a)}.
        \]
        We rearrange this to obtain
        \[
            \frac{N_k(r_1, r_2, r_3, \dots,r_k; a)}
                {N_k(s_1, r_2, r_3, \dots,r_k; a)}
            \le \frac{N_k(r_1, r_2+1, r_3, \dots,r_k; a)}
                {N_k(s_1, r_2+1, r_3, \dots,r_k; a)}.
        \]
        Again, the RHS is obtained from the LHS by increasing the second parameter \(r_2\) by \(1\) while fixing the others. Repeatedly applying this yields that for all \(s_1\ge r_1\) and \(s_2\ge r_2\),
        \[
            \frac{N_k(r_1, \textcolor{blue}{r_2}, r_3, \dots,r_k; a)}
                {N_k(s_1, \textcolor{blue}{r_2}, r_3, \dots,r_k; a)}
            \le \frac{N_k(r_1, \textcolor{purple}{s_2}, r_3, \dots,r_k; a)}
                {N_k(s_1, \textcolor{purple}{s_2}, r_3, \dots,r_k; a)}.
        \]
        Here, the colours assigned to \(r_2\) and \(s_2\) are to highlight the difference between the two sides.
        In fact, by the symmetry of \(N_k(\blank; a)\), it follows that for all \(1\le i<j\le k\), \(s_i\ge r_i\), and \(s_j\ge s_j\),
        \begin{equation}\label{eq:comp-mpart_weak-AF_red2_step1}
            \frac{N_k(r_1,\dots, r_i, \dots, \textcolor{blue}{r_j}, \dots, r_k ; a)}
                {N_k(r_1,\dots, s_i, \dots, \textcolor{blue}{r_j}, \dots, r_k ; a)}
            \le \frac{N_k(r_1,\dots, r_i, \dots, \textcolor{purple}{s_j}, \dots, r_k ; a)}
                {N_k(r_1,\dots, s_i, \dots, \textcolor{purple}{s_j}, \dots, r_k ; a)}.
        \end{equation}
        Repeatedly using \eqref{eq:comp-mpart_weak-AF_red2_step1} gives that, for all \(s_i\ge r_i\), \(1\le i\le k\),
        \begin{align*}
            \frac{N_k(r_1, \textcolor{blue}{r_2}, r_3, r_4 \dots,r_k; a)}
                {N_k(s_1, \textcolor{blue}{r_2}, r_3, r_4 \dots,r_k; a)}
            &\le \frac{N_k(r_1, \textcolor{purple}{s_2}, \textcolor{blue}{r_3}, r_4 \dots,r_k; a)}
                {N_k(s_1, \textcolor{purple}{s_2}, \textcolor{blue}{r_3}, r_4 \dots,r_k; a)}
            \le \frac{N_k(r_1, s_2, \textcolor{purple}{s_3}, \textcolor{blue}{r_4},\dots,r_k; a)}
                {N_k(s_1, s_2, \textcolor{purple}{s_3}, \textcolor{blue}{r_4},\dots,r_k; a)}
            \le \cdots
            \\& \le \frac{N_k(r_1, s_2, s_3, s_4, \dots, s_k; a)}
                {N_k(s_1, s_2, s_3, s_4, \dots, s_k; a)},
        \end{align*}
        which concludes the proof.
    \end{claimproof}

    Finally, \Cref{clm:(comp-mpart_weak-AF)-pf_reduction1} together with \eqref{eq:comp-mpart_weak-AF_stronger_b=a} implies \eqref{eq:comp-mpart_weak-AF_stronger}.
\end{proof}

\section{Concluding remarks}

Having seen the applications of~\Cref{thm:AFM-hom-Lor}, one may wonder whether there are other graphs $H$ than the complete graphs that give a Lorentzian $G$-chromatic function $h_H(\bx;G)$ for antiferromagnetic graphs $G$.
For some particular choices of $G$, this question has already been studied.
For example, Matherne, Morales, and Selover~\cite[Theorem~1.5]{matherne2024newton}
obtained some example graphs $H$ that give 
Lorentzian $h_H(\bx;K_q)$, the $q$-chromatic symmetric functions of $H$.
For another example, real-rootedness of the independence polynomial, a particular evaluation of $h_H(\bx;\indepg)$, of claw-free graphs $H$ shown by Chudnovsky and Seymour~\cite{chudnovsky2007roots} implies that $h_H(\bx;\indepg)$ is Lorentzian for claw-free graphs $H$.

Despite its resemblance to \Mc-convexity shown in~\Cref{prop:reduction}, the fact that \(h_H(\bx;K_q)\) is Lorentzian for every \(q\) does not imply
that \(h_H(\bx;G)\) is Lorentzian for every antiferromagnetic graph~\(G\).
Indeed, when $H=\pathtwo$, all $h_H(\bx;K_q)$ are Lorentzian due to~\cite[Theorem~1.5]{matherne2024newton}, but $h_H(\bx;G)$ is not Lorentzian even for a $3$-vertex antiferromagnetic graph, e.g., \(G=\begin{psmallmatrix}
    0 & 0 & 1 \\
    0 & 0 & 2 \\
    1 & 2 & 0
\end{psmallmatrix}.\) 
This also suggests that one cannot apply the Lorentzian property of $h_H(\bx;G)$ for $H=\pathtwo$ to obtain the cross-bipartite swapping property of {\pathtwo} in $G$.
We believe that complete graphs are the only examples that satisfy the partial derivative condition, i.e., the converse of~\Cref{thm:AFM-hom-Lor} also holds.

\begin{conjecture}
    The $G$-chromatic function \(h_H(\bx;G)\) is Lorentzian for all antiferromagnetic graphs \(G\) if and only if \(H\) is a complete graph.
\end{conjecture}

Our computer search has also checked that for $H=K_{3,3}$, the 3-colour chromatic symmetric polynomial $h_{H}(\bx;K_3)$ is not Lorentzian. Therefore, the cross-bipartite property obtained in~\Cref{sec:complete_multipartite} cannot be derived by using Lorentzian polynomials directly. It would be interesting to see a suitable generalisation of the theory of Lorentzian polynomials that explains our calculation in a conceptual way. 

One may also wonder whether there is a common generalisation of~\Cref{thm:main2,thm:main}. 
We conjecture that all graphs are bipartite swapping in any antiferromagnetic graphs $G$, which implies both \Cref{conj:zhao,conj:SSSZ}.
\begin{conjecture}\label{conj:bipartite_swap}
    Let $G$ be an antiferromagnetic graph. Then for any graph $H$,
    \begin{align*}
        \hom(H,G)^2 \leq \hom(H\times K_2,G).
    \end{align*}
\end{conjecture}
We remark that it is enough to prove this conjecture for non-bipartite graphs $H$. Apart from~\Cref{thm:main2}, the only known non-bipartite instance would be odd cycles $H$. Indeed, it is easy to use the spectrum of $G$ to prove that $\hom(C_{k},G)^2\leq \hom(C_{2k},G)$ for any $k\geq 3$, which implies~\Cref{conj:bipartite_swap} for odd cycles $H$.

As the cross-bipartite swapping property is an interesting correlation inequality for list colourings, it is natural to ask what graphs $H$ are cross-bipartite swapping in an antiferromagnetic graph~$G$, in particular for $G=K_q$.
Although we suspect that all graphs $H$ are cross-bipartite swapping in $K_q$, we do not have a good intuition for what graphs are cross-bipartite swapping in an antiferromagnetic graphs $G$. We leave this as a wide open question.

\begin{question}
    For an antiferromagnetic graph $G$, what graphs are cross-bipartite swapping in $G$?
\end{question}

\noindent\textbf{Acknowledgements.} We are grateful to June Huh for helpful discussions and for bringing~\cite{choe2004homogeneous} to our attention.

\printbibliography

\end{document}